\documentclass[11pt, reqno]{amsart}
\usepackage{amscd}        % Package used to produce simple commutative diagrams
\usepackage[svgnames, dvipsnames, table]{xcolor}
\usepackage[colorlinks = true, allcolors = cyan, pagebackref=true, colorlinks, pdfencoding=auto]{hyperref}

\usepackage{tabu}

\usepackage{orcidlink}

\makeatletter
\newcommand*\rel@kern[1]{\kern#1\dimexpr\macc@kerna}
\newcommand*\widebar[1]{%
  \begingroup
  \def\mathaccent##1##2{%
    \rel@kern{0.8}%
    \overline{\rel@kern{-0.8}\macc@nucleus\rel@kern{0.2}}%
    \rel@kern{-0.2}%
  }%
  \macc@depth\@ne
  \let\math@bgroup\@empty \let\math@egroup\macc@set@skewchar
  \mathsurround\z@ \frozen@everymath{\mathgroup\macc@group\relax}%
  \macc@set@skewchar\relax
  \let\mathaccentV\macc@nested@a
  \macc@nested@a\relax111{#1}%
  \endgroup
}
\makeatother
\usepackage{pdflscape}

\usepackage{float}

\usepackage{nomencl}

\usepackage{algorithm}
\usepackage{algpseudocode}

\usepackage{multirow}

\usepackage{tikz-cd}
\usepackage{graphicx}
\usepackage{rotating}
\usepackage{diagbox}
\usepackage{amssymb}
\usepackage{epstopdf}
\usepackage{tikz}
\definecolor{UCRceleste}{RGB}{0,192,243}
\definecolor{Crimson}{RGB}{220, 20, 60}
\definecolor{Turquoise}{HTML}{04d4f0}
\definecolor{GunmetalGray}{HTML}{1ca6a0}
\definecolor{HotPink}{HTML}{f652a0}
\definecolor{BlueGrotto}{HTML}{059dc0}
\definecolor{GaloisBlue}{HTML}{6495ED}
\definecolor{GaloisRed}{HTML}{DE3163}
\definecolor{mintgreen}{RGB}{152,255,152}
\definecolor{pinksalmon}{RGB}{255,102,102}
\definecolor{hueso}{RGB}{245,245,220}
\definecolor{marfil}{RGB}{255,253,208}
\definecolor{amarillo}{RGB}{255,255,0}
\usetikzlibrary{decorations.markings,arrows}
\usetikzlibrary{decorations.pathreplacing}
\usepackage[inner=1.0in,outer=1.0in,bottom=1.0in, top=1.0in]{geometry}

\numberwithin{equation}{section}
\newtheorem{theorem}{Theorem}[section]
\newtheorem{lemma}[theorem]{Lemma}
\newtheorem{proposition}[theorem]{Proposition}
\newtheorem{corollary}[theorem]{Corollary}

\newtheorem{question}[theorem]{Question}

\makeatletter
\def\moverlay{\mathpalette\mov@rlay}
\def\mov@rlay#1#2{\leavevmode\vtop{%
   \baselineskip\z@skip \lineskiplimit-\maxdimen
   \ialign{\hfil$\m@th#1##$\hfil\cr#2\crcr}}}
\newcommand{\charfusion}[3][\mathord]{
    #1{\ifx#1\mathop\vphantom{#2}\fi
        \mathpalette\mov@rlay{#2\cr#3}
      }
    \ifx#1\mathop\expandafter\displaylimits\fi}
\makeatother

\newcommand{\suchthat}{\;\ifnum\currentgrouptype=16 \middle\fi|\;}

\newcommand{\Z}{\mathbb{Z}}
\newcommand{\C}{\mathbb{C}}
\newcommand{\Q}{\mathbb{Q}}
\newcommand{\R}{\mathbb{R}}

\newcommand{\op}[1]{\operatorname{#1}}

\theoremstyle{definition}
\newtheorem{remark}[theorem]{Remark}

\def\C{{\mathbb C}}
\def\R{{\mathbb R}}
\def\Z{{\mathbb Z}}
\def\Q{{\mathbb Q}}

\def\O_K{{\Cal{O}_{K}}}
\def\O_F{{\Cal{O}_{F}}}
\def\N_F{{\Cal{N}_{F/\Q}}}

\usepackage{listings}
\definecolor{sagebrown}{RGB}{176, 92, 10}
\definecolor{sageblue}{RGB}{44, 45, 254}
\definecolor{sagepurple}{RGB}{151, 57, 164}
\definecolor{sagegreen}{RGB}{18, 103, 68}
\definecolor{sagered}{RGB}{170, 16, 15}
\lstdefinestyle{SageMath-style}{
    backgroundcolor=\color{white},   
    commentstyle=\color{sagebrown},
    keywordstyle=\color{sagepurple},
    keywordstyle = [2]{\color{sageblue}},
    keywordstyle = [3]{\color{yellow}},
    numberstyle=\tiny\color{sagegreen},
    stringstyle=\color{sagered},
    basicstyle=\ttfamily\footnotesize,
    breakatwhitespace=false,         
    breaklines=true,                 
    captionpos=b,                    
    keepspaces=true,                 
    numbers=left,                    
    numbersep=5pt,                  
    showspaces=false,                
    showstringspaces=false,
    showtabs=false,                  
    tabsize=2
}

\begin{document}

\title{The density and distribution of CM elliptic curves over $\Q$}

\author[]{Adrian Barquero-Sanchez\orcidlink{0000-0001-7847-2938} and Jimmy Calvo-Monge\orcidlink{0000-0002-4823-2455}}

\address{Escuela de Matem\'atica, Universidad de Costa Rica, San Jos\'e 11501, Costa Rica}
\address{Centro de Investigación en Matemática Pura y Aplicada, Universidad de Costa Rica, San Jos\'e 2060, Costa Rica}
\email{adrian.barquero\_s@ucr.ac.cr}
\email{jimmy.calvo@ucr.ac.cr}


\begin{abstract}
In this paper we study the density and distribution of CM elliptic curves over $\Q$. In particular, we prove that the natural density of CM elliptic curves over $\Q$, when ordered by naive height, is zero. Furthermore, we analyze the distribution of these curves among the thirteen possible CM orders of class number one. Our results show that asymptotically, $100\%$ of them have complex multiplication by the order $\Z\left[\frac{-1 + \sqrt{-3}}{2} \right]$, that is, have $j$-invariant 0. We conduct this analysis within two different families of representatives for the $\Q$-isomorphism classes of CM elliptic curves: one commonly used in the literature and another constructed using the theory of twists. As part of our proofs, we give asymptotic formulas for the number of elliptic curves with a given $j$-invariant and bounded naive height.
\end{abstract}

\maketitle

\section{Introduction} 

It has long been recognized by experts in the theory of elliptic curves, that the elliptic curves that have complex multiplication are exceedingly rare among all elliptic curves. Nevertheless, these curves possess many properties that make them special and gives them extra structure, which is something that has been exploited many times, for example making the family of CM elliptic curves the first for which more general theorems have been proved; for instance some of the first cases of the Birch and Swinnerton-Dyer conjecture, proved by Coates and Wiles in \cite{CW77} and the proof of modularity of CM elliptic curves by Shimura in \cite{Shi71}. This is why over the past century, through the efforts of many mathematicians, the theory of elliptic curves with complex multiplication has flourished into a rich and very satisfying field, whose elegance was underscored by David Hilbert's famous quote, as recalled by Olga Taussky in Hilbert's obituary in the journal Nature \cite{Tau43}:

\begin{quotation}
It is interesting to recall that, in connection with a lecture by Prof. R. Fueter at the 1932 Zurich Congress, Hilbert asserted that the theory of complex multiplication (of elliptic modular functions) which forms a powerful link between number theory and analysis, is not only the most beautiful part of mathematics but also of all science.
\end{quotation}

In this paper, we aim to investigate the density and distribution of CM elliptic curves over $\mathbb{Q}$ in greater detail from an analytic number theory perspective. To make our study precise, we consider the family $\mathcal{E}$ consisting of all elliptic curves $E_{A, B}$ given by the equation $y^2 = x^3 + Ax + B$, where $A, B \in \mathbb{Z}$ satisfy the condition $\Delta_{E_{A, B}} = -16(4A^3 + 27 B^2) \neq 0$, and there exists no prime number $p$ such that $p^4 | A$ and $p^6 |B$. The last condition on the prime divisors of $A$ and $B$ ensures that each elliptic curve $E/\Q$ is isomorphic over $\Q$ to a unique elliptic curve in the family $\mathcal{E}$. In other words, the family $\mathcal{E}$ serves as a complete set of representatives for the $\Q$-isomorphism classes of elliptic curves over $\Q$ (see e.g.\cite[Proposition 2.19]{Was08}, \cite[p. 45]{Sil09} and \cite[\S 1.3]{Poo13}).

In particular, we want to study how large certain subsets of $\mathcal{E}$ are. To make this precise, we use the notion of natural density, which we define below. First, since the sets that we want to study are infinite, we need a way to bound their size so that we can count finite subsets. We do this by employing the \textit{naive height} 
\begin{align*}
h^{\textrm{naive}}(E_{A, B}) := \max{\{ 4|A|^3, 27 |B|^2 \}}.
\end{align*}
Then, for any $X > 0$ we let $\mathcal{E}(X) := \{ E_{A, B} \in \mathcal{E} \suchthat h^{\textrm{naive}}(E_{A, B}) \leq X \}$.
Now, given any subset $\mathcal{S} \subseteq \mathcal{E}$, we define its \textit{natural density} by
\begin{align*}
d(\mathcal{S}) := \lim_{X \to \infty} \frac{\#(\mathcal{E}(X) \cap \mathcal{S})}{\# \mathcal{E}(X)},
\end{align*}
if the limit exists.

%If $\mathcal{E}^{\textrm{cm}}$ denotes the subset of $\mathcal{E}$ consisting of the elliptic curves with complex multiplication, we abbreviate $\mathcal{E}(X) \cap \mathcal{E}^{\textrm{cm}}$ to $ \mathcal{E}^{\textrm{cm}}(X)$.

Now, let 
\begin{align*}
\mathcal{E}^{\mathrm{cm}} := \{ E_{A, B} \in \mathcal{E} \suchthat \text{$E_{A, B}$ has complex multiplication} \}
\end{align*}
and define $\mathcal{E}^{\textrm{cm}}(X) := \mathcal{E}(X) \cap \mathcal{E}^{\textrm{cm}}$.

\subsection{The density of CM elliptic curves}

As mentioned above, it is widely recognized by experts that the elliptic curves with complex multiplication are very thinly distributed. To illustrate this, we used \texttt{SageMath} \cite{SageMath} to compute the values in Table \ref{table:natural-density-cm-elliptic-curves}.

\begin{table}[H]
{\tabulinesep=1.2mm
\begin{tabu}{c c c c} \hline
$X$ & $\# \mathcal{E}(X)$ & $\# \mathcal{E}^{\textrm{cm}}(X)$ & $\dfrac{\# \mathcal{E}^{\mathrm{cm}}(X)}{\# \mathcal{E}(X)}$ \\ \hline
$10$ & $2$ & $2$ & $1.00000$ \\ 
$100$ & $14$ & $6$ & $0.428571$ \\ 
$1000$ & $166$ & $24$ & $0.144578$ \\ 
$10000$ & $1048$ & $66$ & $0.0629771$ \\ 
$100000$ & $7130$ & $180$ & $0.0252454$ \\ 
$1000000$ & $48070$ & $508$ & $0.0105679$ \\ 
$10000000$ & $329472$ & $1470$ & $0.0044617$ \\ \hline
\end{tabu}}
\caption{The natural density of CM elliptic curves in the family $\mathcal{E}$.}
\label{table:natural-density-cm-elliptic-curves}
\end{table}

These values suggest that the natural density of the set $\mathcal{E}^{\mathrm{cm}}$ inside $\mathcal{E}$ is actually zero. Our first result confirms this expectation.

\begin{theorem}\label{thm:density-of-cm-elliptic-curves}
The natural density of the set of CM elliptic curves $\mathcal{E}^{\mathrm{cm}}$ in the family $\mathcal{E}$ is
\begin{align}
d(\mathcal{E}^{\mathrm{cm}}) = \lim_{X \to \infty} \frac{\#\mathcal{E}^{\mathrm{cm}}(X)}{\# \mathcal{E}(X)} = 0.
\end{align}
\end{theorem}

This result is obtained as a direct consequence of an asymptotic formula for the number of curves in $\mathcal{E}(X)$ due to Brumer \cite[Lemma 4.3]{Bru92} and an asymptotic formula for the number of CM elliptic curves in $\mathcal{E}^{\mathrm{cm}}(X)$. More precisely, Brumer used a slightly different normalization for the naive height than the one we use, so under our normalization of the height his argument gives the following asymptotic formula.

\begin{theorem}[Brumer]\label{thm:brumer}
The number of elliptic curves in $\mathcal{E}(X)$ satisfies the asymptotic formula
\begin{align}\label{eqn:brumer-elliptic-count}
\# \mathcal{E}(X) = \frac{2^{4/3}}{3^{3/2} \zeta(10)} X^{5/6} + O(X^{7/12}),
\end{align}
as $X \to \infty$, where $\zeta(s)$ denotes the Riemann zeta function.
\end{theorem}

In Section \ref{section:distribution-of-cm-elliptic-curves-over-E} we will give Brumer's argument for the asymptotic formula \eqref{eqn:brumer-elliptic-count} under our normalization of the naive height and then we will prove the following theorem.

\begin{theorem}\label{thm:asymptotic-count-cm-elliptic-curves}
The number of CM elliptic curves in $\mathcal{E}^{\mathrm{cm}}(X)$ satisfies the asymptotic formula
\begin{align*}
\#\mathcal{E}^{\mathrm{cm}}(X) = \frac{2}{3^{3/2} \zeta(6)} X^{1/2} + \frac{2^{1/3}}{ \zeta(4)} X^{1/3} + O(X^{1/6}),
\end{align*}
as $X \to \infty$, where $\zeta(s)$ denotes the Riemann zeta function.
\end{theorem}

To illustrate the last two theorems, Table \ref{table:numbers-of-curves-and-the-asymptotic-approximations} presents the computed numbers of curves $\# \mathcal{E}(X)$ and $\# \mathcal{E}^{\mathrm{cm}}(X)$, along with their corresponding asymptotic approximations from Theorems \ref{thm:brumer} and \ref{thm:asymptotic-count-cm-elliptic-curves} for several values of $X$. %As the reader can observe, the asymptotic approximations align closely with the exact computations.

\begin{table}[H]
{\tabulinesep=1.2mm
\begin{tabu}{c | c c | c c} \hline
$X$ & $\# \mathcal{E}(X)$ & $\dfrac{2^{4/3}}{3^{3/2} \zeta(10)} X^{5/6}$ & $\# \mathcal{E}^{\textrm{cm}}(X)$ & $\dfrac{2}{3^{3/2} \zeta(6)} X^{1/2} + \dfrac{2^{1/3}}{ \zeta(4)} X^{1/3}$ \\ \hline
$10$ & $2$ & $3.30060$ & $2$ & $3.70437$ \\ 
$10^2$ & $14$ & $22.4867$ & $6$ & $9.18661$ \\ 
$10^3$ & $166$ & $153.200$ & $24$ & $23.6050$ \\ 
$10^4$ & $1048$ & $1043.74$ & $66$ & $62.9134$ \\ 
$10^5$ & $7130$ & $7110.93$ & $180$ & $173.673$ \\ 
$10^6$ & $48070$ & $48446.2$ & $508$ & $494.747$ \\ 
$10^7$ & $329472$ & $330060.1$ & $1470$ & $1447.207$ \\ \hline
\end{tabu}}
\caption{The numbers of curves $\# \mathcal{E}(X)$ and $\# \mathcal{E}^{\mathrm{cm}}(X)$ and the corresponding asymptotic approximations from Theorems \ref{thm:brumer} and \ref{thm:asymptotic-count-cm-elliptic-curves}.}
\label{table:numbers-of-curves-and-the-asymptotic-approximations}
\end{table}

\subsection{The distribution of CM elliptic curves}

The basic theory of elliptic curves classifies the possibilities for the endomorphism ring $\operatorname{End}(E)$ of an elliptic curve $E/F$ over a field $F$. In particular, if the $\operatorname{char}(F) \neq 0$, the endomorphism ring $\operatorname{End}(E)$ is either isomorphic to $\Z$ or to an order $\mathcal{O}$ in an imaginary quadratic field (see e.g. \cite[Corollary III.9.4]{Sil09}).

We recall that if $K$ is an imaginary quadratic field with ring of integers $\mathcal{O}_K$, an order in $K$ is a subring $\mathcal{O} \subset K$ such that $\mathcal{O}$ is a finitely generated $\Z$-module and $\mathcal{O}$ contains a $\Q$-basis of $K$. Moreover, it is known that an order $\mathcal{O}$ in $K$ is a subring of the ring of integers $\mathcal{O}_K$ and has finite index in $\mathcal{O}_K$. In particular, if $f = [\mathcal{O}_K \colon \mathcal{O}]$ then 
$$
\mathcal{O} = \Z + f \mathcal{O}_K.
$$
The index $f$ is also called the conductor of the order $\mathcal{O}$. A good reference for these basic facts is the beautiful book of Cox \cite[\S 7]{Cox22}. Now, an important result in the theory of complex multiplication of elliptic curves says that if $E/\C$ is an elliptic curve with complex multiplication, then $j(E) \in \widebar{\Z}$ is an algebraic integer and moreover, if $\operatorname{End}(E) \cong \mathcal{O}$ for some order $\mathcal{O}$ in an imaginary quadratic field $K$, then 
\begin{align}\label{eq:j-invariant-degree-formula}
[\Q(j(E)) \colon \Q] = [K(j(E)) \colon K] = h(\mathcal{O}),  
\end{align}
where $h(\mathcal{O}) = \# Cl(\mathcal{O})$, i.e., the class number of the order $\mathcal{O}$ (see e.g. \cite[\S 12.4]{Hus04} or \cite[Theorem II.4.3]{Sil94}). Furthermore, if $E/\Q$ is an elliptic curve with $\operatorname{End}(E) \cong \mathcal{O}$, the degree formula \eqref{eq:j-invariant-degree-formula} implies that $h(\mathcal{O}) = 1$ because $j(E) \in \Q$. Thus, for elliptic curves $E/\Q$ with complex multiplication, the endomorphism ring $\operatorname{End}(E)$ can only be isomorphic to an order $\mathcal{O}$ with class number $h(\mathcal{O}) = 1$.

Now, it is known that $h(\mathcal{O})$ is an integer multiple of $h(\mathcal{O}_K)$ and moreover, that it satisfies the explicit formula
\begin{align}\label{eq:order-class-number-formula}
h(\mathcal{O})=\frac{h(\mathcal{O}_K) f}{\left[\mathcal{O}_K^*: \mathcal{O}^*\right]} \prod_{p \mid f}\left(1-\left(\frac{d_K}{p}\right) \frac{1}{p}\right),
\end{align}
where $d_K$ denotes the discriminant of $K$ and $\mathcal{O}_K^{*}$ and $\mathcal{O}^{*}$ denote the corresponding groups of units (see e.g. \cite[Theorem 7.24]{Cox22}). Hence, given that $h(\mathcal{O})$ is an integer multiple of $h(\mathcal{O}_K)$, a necessary condition for $h(\mathcal{O})$ to be one is that $h(\mathcal{O}_K)$ = 1, that is, the imaginary quadratic field $K$ must have class number 1. In this respect, the famous resolution of Gauss' class number one problem by Heegner in 1952 \cite{Hee52} and later independently in 1966 by Baker \cite{Bak66} and Stark \cite{Sta67a, Sta67b}, determined that there are exactly nine imaginary quadratic fields $K = \Q(\sqrt{d_K})$ with $h(\mathcal{O}_K) = 1$; namely those of discriminant 
$$
d_K \in \{-3,-4,-7,-8,-11,-19,-43,-67,-163 \}.
$$
Then, from this and the explicit formula \eqref{eq:order-class-number-formula}, it can be proved that there are exactly thirteen orders $\mathcal{O}$ of imaginary quadratic fields with class number $h(\mathcal{O}) = 1$ (see e.g. \cite[Theorem 7.30]{Cox22}). More precisely, the thirteen orders of class number one are determined by the first two columns of Table \ref{tab:fixed_cm_curves_table}, which specify the discriminant of the imaginary quadratic field $K$ and the conductor $f$ of the order $\mathcal{O} = \Z + f \mathcal{O}_K$. This table also includes a representative in short Weierstrass form for each $\C$-isomorphism class of CM elliptic curves over $\Q$ and the corresponding LMFDB label \cite{LMFDB}. Moreover, since the $j$-invariant determines the $\C$-isomorphism class, the table also includes the values of the corresponding $j$-invariants, which we shall refer to as the thirteen CM $j$-invariants of elliptic curves over $\Q$ with complex multiplication. See also \cite[p. 483]{Sil94}.

\begin{table}[H]
\begin{tabular}{|c|c|c|c|c|}
\hline
$d_K$ & $f$ & Short Weierstrass equation $E/\mathbb{Q}$ & $j$-invariant of $E$ & LMFDB label \\
\hline \hline
\multirow{3}{*}{$-3$} & 1 & $y^2 = x^3 + 1$ & $0$ & \href{https://www.lmfdb.org/EllipticCurve/Q/36/a/4}{36.a4} \\
& 2 & $y^2 = x^3 - 15x + 22$ & $2^4 \cdot 3^3 \cdot 5^3$ & \href{https://www.lmfdb.org/EllipticCurve/Q/36/a/2}{36.a2} \\
& 3 & $y^2 = x^3 - 120x + 506$ & $-2^{15} \cdot 3 \cdot 5^3 $ & \href{https://www.lmfdb.org/EllipticCurve/Q/1728/n/2}{1728.n2} \\
\hline
\multirow{2}{*}{$-4$} & 1 & $y^2 = x^3 + x$ & $2^6 3^3 = 1728$ & \href{https://www.lmfdb.org/EllipticCurve/Q/64/a/4}{64.a4} \\
& 2 & $y^2 = x^3 - 11x + 14$ & $2^3 \cdot 3^3 \cdot 11^3$ & \href{https://www.lmfdb.org/EllipticCurve/Q/32/a/2}{32.a2} \\
\hline
\multirow{2}{*}{$-7$} & 1 & $y^2 = x^3 - 35x + 98$ & $-3^3 \cdot 5^3$ & \href{https://www.lmfdb.org/EllipticCurve/Q/784/f/4}{784.f4} \\
& 2 & $y^2 = x^3 - 595x + 5586$ & $3^3 \cdot 5^3 \cdot 17^3$ & \href{https://www.lmfdb.org/EllipticCurve/Q/784/f/3}{784.f3} \\
\hline
$-8$ & 1 & $y^2 = x^3 - 30x + 56$ & $2^6 \cdot 5^3$ & \href{https://www.lmfdb.org/EllipticCurve/Q/2304/h/2}{2304.h2} \\
$-11$ & 1 & $y^2=x^3-1056x+13552$ & $-2^{15}$ & \href{https://www.lmfdb.org/EllipticCurve/Q/17424/cb/2}{17424.cb2} \\
$-19$ & 1 & $y^2 = x^3 - 152x + 722$ & $-2^{15} \cdot 3^3$ & \href{https://www.lmfdb.org/EllipticCurve/Q/23104/bc/2}{23104.bc2} \\
$-43$ & 1 & $y^2 = x^3 - 3440x + 77658$ & $-2^{18} \cdot 3^3 \cdot 5^3$ & \href{https://www.lmfdb.org/EllipticCurve/Q/118336/v/2}{118336.v2} \\
$-67$ & 1 & $y^2 = x^3 - 29480x + 1948226$ & $-2^{15} \cdot 3^3 \cdot 5^3 \cdot 11^3$ & \href{https://www.lmfdb.org/EllipticCurve/Q/287296/h/2}{287296.h2} \\
$-163$ & 1 & $y^2 = x^3 - 34790720x - 78984748304$ & $-2^{18} \cdot 3^3 \cdot 5^3 \cdot 23^3 \cdot 29^3$ & \href{https://www.lmfdb.org/EllipticCurve/Q/425104/g/2}{425104.g2} \\
\hline
\end{tabular}
\caption{Representatives in short Weierstrass form for the 13 $\C$-isomorphism classes of CM elliptic curves over $\Q$ with $\operatorname{End}(E) \cong \Z + f\mathcal{O}_K$, where $K = \Q(\sqrt{d_K})$.}
\label{tab:fixed_cm_curves_table}
\end{table}

Even though, as we have already showed in Theorem \ref{thm:density-of-cm-elliptic-curves}, the natural density of CM elliptic curves over $\Q$ is 0, knowing that these CM elliptic curves fall into the thirteen different classes determined by the thirteen orders of class number one, a natural question to ask is the following.

\begin{question}\label{question:distribution}
How are the CM elliptic curves over $\Q$ distributed among the thirteen possible CM orders of class number one?
\end{question}

To investigate this question, we introduce the following notation. For any $j \in \Q$, let
\begin{align}\label{def:E_j}
    \mathcal{E}_j := \{ E_{A,B} \in \mathcal{E} \suchthat j(E_{A,B}) = j\}.
\end{align}
In other words, $\mathcal{E}_j$ is the subset of all elliptic curves in $\mathcal{E}$ with $j$-invariant equal to $j$. Moreover, let $\mathcal{J}^{\textrm{cm}}$ be the set of all the possible thirteen $j$-invariants of CM elliptic curves over $\Q$, as appearing in Table \ref{tab:fixed_cm_curves_table}. This set is given explicitly by
\begin{equation}
\begin{aligned}\label{eqn:set-of-cm-j-invariants}
\mathcal{J}^{\textrm{cm}} = \{
& 0, 1728, -3375, 8000, -32768, 54000, 287496, -12288000,\\
& 16581375, -884736, -884736000, -147197952000, -262537412640768000\}.
\end{aligned}
\end{equation}

For a given $j \in \mathcal{J}^{\textrm{cm}}$, every elliptic curve $E_{A, B} \in \mathcal{E}_j$ has complex multiplication by the order corresponding to the given CM $j$-invariant. Hence, for a CM $j$-invariant $j \in \mathcal{J}^{\textrm{cm}}$ we have
$$
\mathcal{E}_j \subset \mathcal{E}^{\textrm{cm}}.
$$
%\begin{align}
%    \mathcal{E}_j := \{ E_{A,B} \in \mathcal{E}^{\mathrm{cm}} \suchthat j(E_{A,B}) = j\}.
%\end{align}

We will also let $\mathcal{E}_j(X) := \mathcal{E}_j \cap \mathcal{E}(X)$. Then, the classification of CM elliptic curves over $\Q$ implies that we can decompose the families $\mathcal{E}^{\textrm{cm}}$ and $\mathcal{E}^{\textrm{cm}}(X)$ as disjoint unions given by
\begin{align}\label{eqn:decomposition_cm_elliptic_curves}
    \mathcal{E}^{\mathrm{cm}} = \coprod_{j \in \mathcal{J}^{\textrm{cm}}} \mathcal{E}_j \quad \text{and} \quad \mathcal{E}^{\mathrm{cm}}(X) = \coprod_{j \in \mathcal{J}^{\textrm{cm}}} \mathcal{E}_j(X)
\end{align}
for every $X > 0$.

Having set up these definitions, we now address Question \ref{question:distribution}. Just like we did with the question of the density of CM elliptic curves, we first explored Question \ref{question:distribution} by conducting some numerical investigations using \texttt{SageMath}. The resulting data is presented in Table \ref{table:densities-cm-j-invariants}. Interestingly, this data suggests that although for each of the thirteen CM orders listed in Table \ref{table:densities-cm-j-invariants} there are infinitely many elliptic curves over $\mathbb{Q}$ with complex multiplication by that order, those with $j$-invariant 0—i.e., those with $\operatorname{End}(E) \cong \mathbb{Z}\left[ \frac{-1 + \sqrt{-3}}{2} \right]$—seem to occur much more frequently. Indeed, Table \ref{table:density-j-0} illustrates the trend of the proportion of CM elliptic curves with $j$-invariant 0 and bounded height as the bound on the height increases. We confirm what the trend in Table \ref{table:density-j-0} suggests in the next theorem.

\begin{theorem}\label{thm:density-j-0}
The elliptic curves with $j$-invariant 0 comprise $100\%$ of all the CM elliptic curves in $\mathcal{E}^{\textrm{cm}}$. More precisely, the natural density of the family $\mathcal{E}_0$ in $\mathcal{E}^{\mathrm{cm}}$ equals $1$, that is
\begin{align*}
d(\mathcal{E}_0) = \lim_{X\to \infty} \frac{\#\mathcal{E}_0(X)}{\#\mathcal{E}^{\mathrm{cm}}(X)} = 1.
\end{align*}
\end{theorem}

This theorem will be proved in Section \ref{section:distribution-of-cm-elliptic-curves-over-E} and it will be a consequence of the next result, which gives asymptotic formulas for the number of elliptic curves inside the family $\mathcal{E}$ with a given $j$-invariant and bounded naive height.

\begin{theorem}\label{thm:asymptotic-for-curves-with-given-j}
Let $j \in \Q$. Then
\begin{align}\label{eqn:Ej-asymptotic}
    \#\mathcal{E}_j(X) = O(X^{1/m(j)}),
\end{align}
where
\begin{align*}
    m(j) =
        \begin{cases}
            6 & \text{if } j\neq 0, 1728, \\
            3 & \text{if } j = 1728, \\
            2 & \text{if } j = 0.
        \end{cases}
\end{align*}
Moreover, for $j = 0$ and $j = 1728$ we have the asymptotic formulas
\begin{align}
    \#\mathcal{E}_0(X) = \frac{2}{3^{3/2} \zeta(6)} X^{1/2} + O(X^{1/12})  \quad \text{and} \quad \#\mathcal{E}_{1728}(X) = \frac{2^{1/3}}{ \zeta(4)} X^{1/3} + O(X^{1/12})
\end{align}
as $X \to \infty$. In particular, these asymptotic formulas combined imply that the number of CM elliptic curves in $\mathcal{E}^{\mathrm{cm}}(X)$ satisfies the asymptotic formula
\begin{align}\label{eqn:ECM-asymptotic}
\#\mathcal{E}^{\mathrm{cm}}(X) = \frac{2}{3^{3/2} \zeta(6)} X^{1/2} + \frac{2^{1/3}}{ \zeta(4)} X^{1/3} + O(X^{1/6}),
\end{align}
as $X \to \infty$.
\end{theorem}

\begin{remark}
A different way of studying the distribution of CM elliptic curves over $\Q$ among the thirteen different orders of class number 1 is explored in Section \ref{section:distribution-of-cm-elliptic-curves-using-twists}. This is done in terms of the theory of twists. In particular, in that setting we are able to give a leading term for the count of curves with $j$-invariant $j \neq 0, 1728$ in Theorem \ref{thm:cm-elliptic-curves-asymptotic-twist-count}.
\end{remark}

\begin{table}[H]
{\tabulinesep=1.2mm
\begin{tabu}{c c c c c} \hline
$d_K$ & $f$ & $j$-invariant & $\mathcal{E}_{j}(10^{10})$ & $\dfrac{\mathcal{E}_{j}(10^{10})}{\mathcal{E}^{\text{cm}}(10^{10})}$ \\ \hline
\multirow{3}{*}{$-3$} & $1$ & $0$ & $37836$ & $0.936303$ \\
& $2$ & $2^4 \cdot 3^3 \cdot 5^3$ & $12$ & $0.000296956$ \\
& $3$ & $-2^{15} \cdot 3 \cdot 5^3$ & $6$ & $0.000148478$ \\ \hline
\multirow{2}{*}{$-4$} & $1$ & $2^6 \cdot 3^3 = 1728$ & $2512$ & $0.0621628$ \\
& $2$ & $2^3 \cdot 3^3 \cdot 11^3$ & $16$ & $0.000395942$ \\ \hline
\multirow{2}{*}{$-7$} & $1$ & $-3^3 \cdot 5^3$ & $8$ & $0.000197971$ \\
& $2$ & $3^3 \cdot 5^3 \cdot 17^3$ & $2$ & $0.0000494927$ \\ \hline
$-8$ & $1$ & $2^6 \cdot 5^3$ & $10$ & $0.000247463$ \\
$-11$ & $1$ & $-2^{15}$ & $4$ & $0.0000989854$ \\
$-19$ & $1$ & $-2^{15} \cdot 3^3$ & $4$ & $0.0000989854$ \\
$-43$ & $1$ & $-2^{18} \cdot 3^3 \cdot 5^3$ & $0$ & $0$ \\
$-67$ & $1$ & $-2^{15} \cdot 3^3 \cdot 5^3 \cdot 11^3$ & $0$ & $0$ \\
$-163$ & $1$ & $-2^{18} \cdot 3^3 \cdot 5^3 \cdot 23^3 \cdot 29^3$ & $0$ & $0$ \\ \hline \\
\end{tabu}}
\caption{The proportions of elliptic curves in the family $\mathcal{E}^{\mathrm{cm}}$ with CM by the different orders of class number 1 and naive height $h^{\textrm{naive}}(E) \leq 10^{10}$.}
\label{table:densities-cm-j-invariants}
\end{table}

\begin{table}[H]
{\tabulinesep=1.2mm
\begin{tabu}{c c c c c} \hline
$X$ & $\#\mathcal{E}^{\textrm{cm}}(X)$ & $\#\mathcal{E}_0(X)$ & $\dfrac{2}{3^{3/2} \zeta(6)} X^{1/2}$  & $\dfrac{\mathcal{E}_{0}(X)}{\mathcal{E}^{\text{cm}}(X)}$ \\ \hline
$10^2$ & $6$ & $2$ & $3.7833863$ & $0.33333$ \\
$10^3$ & $24$ & $12$ & $11.964118$ & $0.5$ \\
$10^4$ & $66$ & $38$ & $37.833863$ & $0.57576$ \\
$10^5$ & $180$ & $120$ & $119.64118$ & $0.66667$ \\
$10^6$ & $508$ & $378$ & $378.33863$ & $0.74409$ \\ 
$10^7$ & $1470$ & $1198$ & $1196.4118$ & $0.81497$ \\
$10^8$ & $4356$ & $3784$ & $3783.3863$ & $0.86869$\\
$10^9$ & $13174$ & $11964$ & $11964.118$ & $0.90815$\\
$10^{10}$ & $40410$ & $37836$ & $37833.863$ & $0.93630$\\
$10^{11}$ & $125336$ & $119646$ & $119641.18$ & $0.95460$\\
$10^{12}$ & $390312$ & $378342$ & $378338.63$ & $0.96933$ \\
\hline \\
\end{tabu}}
\caption{The natural density of elliptic curves with CM having $j$-invariant 0 inside the family $\mathcal{E}^{\textrm{cm}}$ with naive height $h^{\textrm{naive}}(E) \leq X$ and the asymptotic for $\#\mathcal{E}_0(X)$.}
\label{table:density-j-0}
\end{table}

%\subsection{Future and ongoing work}
%Future work will focus on extending the density results established for CM elliptic curves over $\Q$ to broader settings, particularly to CM elliptic curves defined over various number fields. This generalization involves investigating the natural density and distribution properties of these curves when ordered by appropriate height functions over number fields. The approach using the theory of twists, which proved effective in our analysis over $\Q$, appears promising for this purpose. Preliminary results suggest that similar asymptotic behavior may be observed, but with subtle differences arising from the arithmetic and geometric properties unique to each number field. This is work in progress by the authors.

\subsection{Organization of the paper}

The paper is organized as follows. In Section \ref{section:distribution-of-cm-elliptic-curves-over-E} we prove Theorems \ref{thm:density-of-cm-elliptic-curves}, \ref{thm:brumer}, \ref{thm:asymptotic-count-cm-elliptic-curves}, \ref{thm:density-j-0} and \ref{thm:asymptotic-for-curves-with-given-j}. Then, in Section \ref{section:twists} we recall some basic facts about twists of elliptic curves and in Section \ref{section:distribution-of-cm-elliptic-curves-using-twists} we use the theory of twists to give an alternative way of studying the distribution of CM elliptic curves over $\Q$ among the thirteen CM orders of class number one.

%%%%%%%%%%%%%%%%%%%%%%%%%%%%%%%%%%%%%%%%%%%%%%%%%%%%%%%%%%%%%
%%%%%%%%%%%%%%%%%%%%%%%%%%%%%%%%%%%%%%%%%%%%%%%%%%%%%%%%%%%%%
%%%%%%%%%%%%%%% Distribution in E section %%%%%%%%%%%%%%%%%%%
%%%%%%%%%%%%%%%%%%%%%%%%%%%%%%%%%%%%%%%%%%%%%%%%%%%%%%%%%%%%%
%%%%%%%%%%%%%%%%%%%%%%%%%%%%%%%%%%%%%%%%%%%%%%%%%%%%%%%%%%%%%
\section{The distribution of CM elliptic curves over \texorpdfstring{$\Q$}{Q} in the family \texorpdfstring{$\mathcal{E}$}{E} }\label{section:distribution-of-cm-elliptic-curves-over-E}

As stated in the introduction, Brumer in \cite[Lemma 4.3]{Bru92} provided a count of elliptic curves in the family $\mathcal{E}$ with height bounded by $X$ using a different normalization for the naive height than the one we use. For the reader's convenience, we adapt Brumer's argument to align with our height normalization. This adjustment not only ensures consistency within our framework but also we describe in more detail some ideas that lay the groundwork for our subsequent enumeration of CM elliptic curves.

\begin{theorem}
The number of elliptic curves in $\mathcal{E}(X)$ satisfies the asymptotic formula
\begin{align}\label{eqn:elliptic-count}
\# \mathcal{E}(X) = \frac{2^{4/3}}{3^{3/2} \zeta(10)} X^{5/6} + O(X^{7/12}),
\end{align}
as $X \to \infty$, where $\zeta(s)$ denotes the Riemann zeta function.
\end{theorem}

\begin{proof}
First we define some families of curves. Recall that for any point $(A, B) \in \Z^2$ we define the curve $E_{A, B} \colon y^2 = x^3 + Ax + B$. Then, for $X > 0$ let
\begin{align*}
\mathcal{D}(X) &= \{ E_{A, B} \suchthat (A, B) \in \Z^2, \, h^{\textrm{naive}}(E_{A, B}) \leq X \}, \\
\mathcal{D}^{\prime}(X) &= \{ E_{A, B} \in \mathcal{D}(X) \suchthat (A, B) \in \Z^2 \smallsetminus \{ (0, 0) \} \}, \\
\mathcal{M}(X) &= \{ E_{A, B} \in \mathcal{D}^{\prime}(X) \suchthat \text{there's no prime $p \in \Z$ such that $p^4 \mid A$ and $p^6 \mid B$} \}, \text{and}\\
\mathcal{S}(X) &= \{ E_{A, B} \in \mathcal{M}(X) \suchthat \text{$E_{A, B}$ is singular} \}.
\end{align*}
Note then that $\mathcal{E}(X) \subseteq \mathcal{M}(X) \subseteq \mathcal{D}^{\prime}(X) \subseteq \mathcal{D}(X)$ and moreover we have
$$
\mathcal{E}(X) = \{ E_{A, B} \in \mathcal{M}(X) \suchthat \Delta_{E_{A, B}} \neq 0 \}. 
$$
Next, for $d \in \Z$, we define $d * E_{A, B} := E_{d^4A, d^6 B}$. Observe in particular that from the definition of the naive height we have
$$
h^{\textrm{naive}}(d * E_{A, B}) = d^{12} \cdot h^{\textrm{naive}}(E_{A, B}).
$$
Therefore $h^{\textrm{naive}}(d * E_{A, B}) \leq X$ if and only if $h^{\textrm{naive}}(E_{A, B}) \leq X/d^{12}$. Using this, it follows that every curve in the family $\mathcal{D}^{\prime}(X)$ can be written uniquely as a twist $d * E_{A, B}$ for a unique curve $E_{A, B} \in \mathcal{M}(d^{-12} X)$. Therefore we have
\begin{align}\label{eqn:curve-families-decomposition}
\mathcal{D}^{\prime}(X) = \coprod_{d = 1}^{\lfloor X^{1/12} \rfloor} d * \mathcal{M}(d^{-12} X) \quad \text{and}\quad \mathcal{M}(X) = \mathcal{E}(X) \coprod \mathcal{S}(X).
\end{align}

We note that $h^{\textrm{naive}}(E_{A, B}) \leq X$ if and only if $|A| \leq \dfrac{X^{1/3}}{2^{2/3}}$ and $|B| \leq \dfrac{X^{1/2}}{3^{3/2}}$. Hence, the bound on the naive height $h^{\textrm{naive}}(E_{A, B}) \leq X$ is equivalent to the lattice point $(A, B) \in \Z^2$ lying on a rectangular box centered at the origin.
%, as seen in Figure \ref{fig:box}.

Now, for the singular curves in $\mathcal{M}(X)$ we claim that
\begin{align}\label{eq:singular-curves-bound}
\# \mathcal{S}(X) = O(X^{1/6}).
\end{align}
Indeed, a curve $E_{A,B}$ is singular if and only if $\Delta_{E_{A,B}}=-16(4A^3+27B^2) = 0$. Then, if $\Delta_{E_{A,B}} = 0$, the number $A$ must be negative and moreover, the real points $(a, b) \in \R^2$ such that $4a^3+27b^2 = 0$ can be parametrized by $a = -3w^2$ and $b = 2w^3$ for $w \in \R$. Then, using this we can easily see that we must have $|w| \leq \dfrac{X^{1/6}}{2^{1/3} 3^{1/2}}$ to ensure $h^{\textrm{naive}}(E_{A, B}) \leq X$ and that $w$ must belong to $\Z$ to have $(a,b) \in \Z \times \Z$. Then, counting the number of integers $w \in \Z$ satisfying the inequality $|w| \leq \dfrac{X^{1/6}}{2^{1/3} 3^{1/2}}$, we conclude that $\# \mathcal{S}(X) \leq 2\dfrac{X^{1/6}}{2^{1/3} 3^{1/2}} = \dfrac{2^{2/3}} {3^{1/2}} X^{1/6}$, which proves the claim \eqref{eq:singular-curves-bound}.

% \begin{figure}[H]
% \begin{tikzpicture}[scale=0.3]

% % Draw axes
% \draw[thick, ->] (-13,0) -- (13,0) node[right] {$x$};
% \draw[thick, ->] (0,-17) -- (0,17) node[above] {$y$};

% % Draw rectangle
% \draw[thick] (-11.447, -14.907) rectangle (11.447, 14.907);

% % Plot implicit curves

% % Using j=-3375
% \draw[very thick, GaloisRed, domain=-11:-0.005, samples=50] plot (\x, {sqrt(-28/125*\x^3)});
% \draw[very thick, GaloisRed, domain=-11:-0.005, samples=50] plot (\x, {-sqrt(-28/125*\x^3)});

% % Using j=8000
% \draw[very thick, blue, domain=-12:-0.005, samples=50] plot (\x, {sqrt(-392/3375*\x^3)});
% \draw[very thick, blue, domain=-12:-0.005, samples=50] plot (\x, {-sqrt(-392/3375*\x^3)});

% % Using j=-262537412640768000 d = - 163
% %\draw[very thick, orange, domain=-12:-0.005, samples=50] plot (\x, {sqrt(-0.148148148148149*\x^3)});
% %\draw[very thick, orange, domain=-12:-0.005, samples=50] plot (\x, {-sqrt(-0.148148148148149*\x^3)});

% % Draw grid
% \foreach \x in {-11,-10,...,11}
%     \foreach \y in {-14,-13,...,14}
%         \fill[GaloisBlue] (\x,\y) circle (0.2);

% \end{tikzpicture}
% \caption{The box determined by the condition $h^{\textrm{naive}}(E_{A, B}) \leq X$ for $X = 6000$ and the points $(A, B) \in \Z^2$ contained within the box.}
% \label{fig:box}
% \end{figure}

Next, observe that $\# \mathcal{D}^{\prime}(X)$ is equal to the number of lattice points $(A, B) \in \Z^2 \smallsetminus{\{ (0, 0) \}}$ within the box determined by the previous inequalities on $|A|$ and $|B|$. Thus we have
\begin{align}\label{eqn:D-prime-count}
\# \mathcal{D}^{\prime}(X) = \left( 2 \left \lfloor \frac{X^{1/3}}{2^{2/3}}  \right \rfloor + 1 \right) \left( 2 \left \lfloor \frac{X^{1/2}}{3^{3/2}}  \right \rfloor + 1 \right) - 1.
\end{align}

Now, a form of the Möbius inversion formula is the following. If $F(x)$ and $G(x)$ are complex valued functions defined in the interval $[1, \infty)$ such that
$$
G(x) = \sum_{1 \leq n \leq x} F \left( \frac{x}{n} \right)
$$
for every $x \geq 1$, then
$$
F(x) = \sum_{1 \leq n \leq x} \mu(n) G \left( \frac{x}{n} \right)
$$
for every $x \geq 1$. Therefore, from the decomposition \eqref{eqn:curve-families-decomposition} we have
\begin{align*}
\# \mathcal{D}^{\prime}(X) = \sum_{1 \leq d \leq X^{1/12}} \# (d * \mathcal{M}(d^{-12}X)) = \sum_{1 \leq d \leq X^{1/12}} \# \mathcal{M}(d^{-12}X),
\end{align*}
and using the Möbius inversion formula just cited, we get
\begin{align*}
    \#\mathcal{M}(X) &= \sum_{1 \leq d \leq X^{1/12}} \mu(d) \# \mathcal{D}^{\prime}(d^{-12}X) \\
    &= \sum_{1 \leq d \leq X^{1/12}} \mu(d) \left\{  \left( 2 \left \lfloor \frac{(d^{-12}X)^{1/3}}{2^{2/3}}  \right \rfloor + 1 \right) \left( 2 \left \lfloor \frac{(d^{-12}X)^{1/2}}{3^{3/2}}  \right \rfloor + 1 \right) - 1  \right\}\\
    &= \sum_{1 \leq d \leq X^{1/12}} \mu(d) \left\{  \left( 2 \left \lfloor \frac{X^{1/3}}{d^{4} 2^{2/3}}  \right \rfloor + 1 \right) \left( 2 \left \lfloor \frac{X^{1/2}}{d^{6} 3^{3/2}}  \right \rfloor + 1 \right) - 1  \right\}.
\end{align*}
Now, using the fact that $\lfloor X \rfloor = X + O(1)$ for every $X$, we get
\begin{align}\label{eq:M(X)-expression}
\#\mathcal{M}(X) &= \sum_{1 \leq d \leq X^{1/12}} \mu(d) \left\{  \left( 2  \frac{X^{1/3}}{d^{4} 2^{2/3}}  + O(1) \right) \left( 2 \frac{X^{1/2}}{d^{6} 3^{3/2}}  + O(1) \right) - 1  \right\} \notag \\
&= \sum_{1 \leq d \leq X^{1/12}} \mu(d) \left( 4  \frac{X^{5/6}}{d^{10} 2^{2/3} 3^{3/2}} + O(X^{1/2})\right) \notag \\
&= \left( \frac{2^{4/3}}{3^{3/2}} \sum_{1 \leq d \leq X^{1/12}} \frac{\mu(d)}{d^{10}} \right) X^{5/6} + \left( \sum_{1 \leq d \leq X^{1/12}} \mu(d) \right) O(X^{1/2}).
\end{align}

We will now give an asymptotic formula for the expression inside the first parentheses in \eqref{eq:M(X)-expression} and a bound for the expression inside the second parentheses. Thus, for the expression inside the first parentheses we have
\begin{align*}
    \frac{2^{4/3}}{3^{3/2}} \sum_{1 \leq d \leq X^{1/12}} \frac{\mu(d)}{d^{10}} &= \frac{2^{4/3}}{3^{3/2}} \left( \sum_{d=1}^{\infty} \frac{\mu(d)}{d^{10}} - \sum_{d >X^{1/12}} \frac{\mu(d)}{d^{10}}\right) \\ &= 
    \frac{2^{4/3}}{3^{3/2}\zeta(10)} - \frac{2^{4/3}}{3^{3/2}}\sum_{d >X^{1/12}} \frac{\mu(d)}{d^{10}},
\end{align*}
where we used the well known formula 
\begin{align*}
    \frac{1}{\zeta(s)} = \sum_{n = 1}^{\infty} \frac{\mu(n)}{n^s}
\end{align*}
for the reciprocal of the Riemann zeta function.

Next, using the asymptotic bound (see e.g. \cite[Theorem 3.2]{Apostol76})
\begin{align*}
    \sum_{n>x}\frac{1}{n^s} = O(x^{1-s}), \quad \text{for } s>1 \text{ and } x\geq 1,
\end{align*}
we see that
\begin{align*}
    \left|\sum_{d >X^{1/12}} \frac{\mu(d)}{d^{10}} \right| \leq \sum_{d >X^{1/12}} \frac{|\mu(d)|}{d^{10}} \leq \sum_{d >X^{1/12}} \frac{1}{d^{10}} = O(X^{-9/12}).
\end{align*}
Hence we have
\begin{align}\label{eq:M(X)-expression-first-term-asymptotic}
 \frac{2^{4/3}}{3^{3/2}} \sum_{1 \leq d \leq X^{1/12}} \frac{\mu(d)}{d^{10}} = \frac{2^{4/3}}{3^{3/2}\zeta(10)} + O(X^{-9/12}).
\end{align}

Now, for the expression inside the second parentheses in \eqref{eq:M(X)-expression} we have the trivial bound
\begin{align}\label{eq:mu-sum-trivial-bound}
    \sum_{1 \leq d \leq X^{1/12}} \mu(d) = O(X^{1/12}).
\end{align}
Therefore, substituting \eqref{eq:M(X)-expression-first-term-asymptotic} and \eqref{eq:mu-sum-trivial-bound} in \eqref{eq:M(X)-expression} we obtain
\begin{align}\label{eq:M(X)-asymptotic}
\#\mathcal{M}(X) &= \left( \frac{2^{4/3}}{3^{3/2}\zeta(10)} + O(X^{-9/12}) \right) X^{5/6} +  O(X^{1/12}) \cdot O(X^{1/2}) \notag\\
&= \frac{2^{4/3}}{3^{3/2}\zeta(10)} X^{5/6} + O(X^{7/12})
\end{align}
Finally, from \eqref{eqn:curve-families-decomposition} we have that 
\begin{align}
\#\mathcal{E}(X) = \#\mathcal{M}(X) - \#\mathcal{S}(X).
\end{align}
Consequently, combining this formula with the asymptotic formulas \eqref{eq:M(X)-asymptotic} and \eqref{eq:singular-curves-bound} we obtain
\begin{align*}
    \# \mathcal{E}(X) = \frac{2^{4/3}}{3^{3/2} \zeta(10)} X^{5/6} + O(X^{7/12}),
\end{align*}
which proves the theorem.
\end{proof}

\subsection{The number of curves with a given \texorpdfstring{$j$}{j}-invariant}

Now we are going to count how many curves $E_{A, B} \in \mathcal{E}(X)$ have a given $j$-invariant. From this we will derive the counts for CM elliptic curves as described in the introduction. As part of the proofs we will need to estimate how many integers in an interval are $k$-free. We recall that an integer $n$ is said to be $k$-free (for $k \geq 2$) if there is no prime $p \in \Z$ such that $p^k \mid n$, that is, if $n$ is not divisible by a perfect $k$-th power greater than 1. We will use the following well known result, which appears to have been first proved by Gegenbauer in \cite{Gege85}. A more modern reference is the book of Montgomery and Vaughan \cite[Proposition 2.2]{MV07}. 

\begin{proposition}\label{thm:density-k-free-integers}
For $X\in \R_{\geq 0}$ and $k \geq 2$ an integer, let $Q_k(X)$ denote the number of positive $k$-free integers $n$ such that $n \leq X$. Then
\begin{align*}
    Q_k(X) = \frac{1}{\zeta(k)}X + O(X^{1/k})
\end{align*}
as $X \to \infty$, where $\zeta(s)$ is the Riemann zeta function.
\end{proposition}

Moreover, in our arguments we will also need the following lemma.

\begin{lemma}\label{lemma:technical_lemma} 
Let $a = \frac{p}{q} \in \Q^{\times}$ with $p, q \in \Z$ coprime integers and consider the cuspidal cubic $\mathcal{C}_a: y^2 = a x^3$. Moreover, let $\mathcal{C}_a(\Z)$ be the set of integral points on $\mathcal{C}_a$, that is, the set of points $(x_0, y_0) \in \Z^2$ such that $y_0^2 = a x_0^3$. Then for every $T > 0$ we have
\begin{align}\label{eq:integral-points-bound}
    \# \{ (x_0, y_0) \in \mathcal{C}_a(\Z) \suchthat |x_0| \leq T \}  \leq 2\sigma_0(q)\sqrt{|pq|} T^{1/2},
\end{align}
where $\sigma_0(q) := \sum \limits_{d \mid q} 1$ is the number of positive divisors of $q$. In other words, we have
\begin{align}\label{eq:bound_lemma_1}
    \# \{ (x_0, y_0) \in \mathcal{C}_a(\Z) \suchthat |x_0| \leq T \} = O(T^{1/2}).
\end{align}
\end{lemma}

\begin{proof}
By projecting from the double point at the origin in the curve $\mathcal{C}_a$ onto the vertical line $x = 1$ as in Figure \ref{fig:cuspidal-cubic}, we obtain the rational parametrization $\left( x(t), y(t) \right) = \left( \dfrac{t^2}{a}, \dfrac{t^3}{a} \right)$ for $t \in \R$.

\begin{figure}[H]
\begin{tikzpicture}[scale=1]
% Draw axes
\draw[->,thick] (-3,0) -- (3,0) node[right] {$x$};
\draw[->,thick] (0,-4) -- (0,4) node[above] {$y$};

% Plot parametric function
\draw[red,very thick] plot[smooth,variable=\t,domain=-1.6:1.6] ({-pow(\t,2)},{pow(\t,3)}) node[right] {$\mathcal{C}_a \colon y^2 = ax^3$};
% Vertical line x=1
\draw[dashed, very thick] (1,-4) -- (1,4) node[above] {$x=1$};
% Line y = 1.4x and intersection point
\draw[cyan, very thick] plot[domain=-2.5:2] (\x,{1.4*\x}) node[right] {$y=tx$};
\filldraw[black] (1,1.4) circle [radius=0.07] node[right] {$(1,t)$};
\filldraw (0,0) circle [radius=0.07];
\filldraw (-1.96, -2.744) circle [radius=0.07] node[right] {$\left(\dfrac{t^2}{a}, \dfrac{t^3}{a} \right)$};
\end{tikzpicture} 
\caption{A rational parametrization of the cuspidal cubic $\mathcal{C}_a \colon y^2 = ax^3$. The graph shows a typical curve with $a < 0$.}
\label{fig:cuspidal-cubic}
\end{figure}
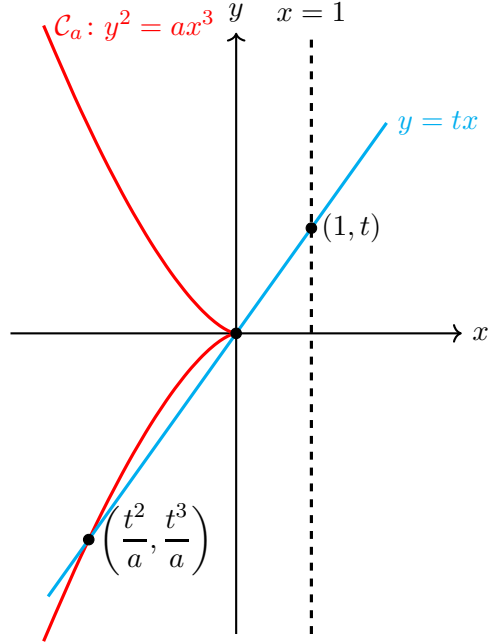

Indeed, note that when we substitute $y = tx$ into the equation for $\mathcal{C}_a$ we get
\begin{align*}
t^2 x^2 = ax^3 \iff x^2 (ax - t^2) = 0 \iff x = 0 \quad \text{or} \quad x = \dfrac{t^2}{a}.  
\end{align*}

This parametrization gives us a bijection

\begin{align}
     \varphi \colon &\R \longrightarrow \mathcal{C}_a \nonumber \\
     &t \longmapsto P_{t} := \left( \frac{t^2}{a}, \frac{t^3}{a} \right).
\end{align}

Moreover, note that restricting $\varphi$ to $\Q$ also gives us a bijection $\widetilde{\varphi} \colon \Q \longrightarrow C_a(\Q)$, where $\mathcal{C}_a(\Q)$ denotes the set of rational points on $\mathcal{C}_a$.

Suppose now that $(x_0, y_0) \in \mathcal{C}_a(\Z)$ is an integral point on $\mathcal{C}_a$ and let $t_0 \in \Q$ satisfy $\varphi(t_0) = (x_0, y_0)$. If $x_0 = 0$ then $t_0=0$, so we assume $x_0 \neq 0$. If we write $a = \dfrac{p}{q}$ and $t_0 = \dfrac{r_0}{s_0}$ with $p, q, r_0, s_0\in \Z$ such that $\op{gcd}(p,q)=1$ and $\op{gcd}(r_0, s_0) = 1$, then the relation $\varphi(t_0) = (x_0, y_0)$ implies that
\begin{align*}
\frac{q}{p} \cdot \frac{r_0^2}{s_0^2} = x_0 \quad \text{and} \quad \frac{q}{p} \cdot \frac{r_0^3}{s_0^3} = y_0,
\end{align*}
or equivalently, that
\begin{equation}\label{eq:cuspidal-diophantine}
\begin{aligned}
qr_0^2 &= x_0ps_0^2 \\
qr_0^3 &= y_0ps_0^3.
\end{aligned}
\end{equation}

Then, using the coprimality condition $\op{gcd}(r_0,s_0)=1$, the first equation in \eqref{eq:cuspidal-diophantine} implies that $s_0^2$ divides $q$ and hence $s_0$ divides $q$. Therefore we have

\begin{align}\label{eq:count_lemma_1}
    & \# \{ (x_0, y_0) \in \mathcal{C}_a(\Z) \suchthat |x_0| \leq T \} \nonumber \\
    & \leq 
    \# \left\{ t_0 = \frac{r_0}{s_0} \suchthat \text{$r_0, s_0 \in \Z$ coprime, with $s_0 \mid q$ and $|t_0| \leq |a|^{1/2} T^{1/2}$} \right\}.
\end{align}
Then, for each integer $s_0$ with $s_0|q$ we have that 
\begin{align*}
    |t_0| \leq |a|^{1/2} T^{1/2} \implies |r_0| \leq |s_0| |a|^{1/2} T^{1/2} \leq q |a|^{1/2} T^{1/2} = \sqrt{|pq|} T^{1/2}.
\end{align*}
Therefore, applying this to every positive divisor of $q$, we use \eqref{eq:count_lemma_1} to conclude that
\begin{align}
    & \# \{ (x_0, y_0) \in \mathcal{C}_a(\Z) \suchthat |x_0| \leq T \} \nonumber \\
    & \leq  \sigma_0(q) \cdot \#\left\{ r_0 \in \Z \setminus \{0\} \suchthat |r_0| \leq \sqrt{|pq|} T^{1/2} \right\} \leq 2\sigma_0(q)\sqrt{|pq|} T^{1/2}.
\end{align}
We then have proved that
\begin{align*}
    \# \{ (x_0, y_0) \in C_a(\Z) \suchthat |x_0| \leq T \} = O(T^{1/2}).
\end{align*}

\end{proof}

Now we recall the following notation from the introduction. For any $j \in \Q$, let
\begin{align*}
    \mathcal{E}_j := \{ E_{A,B} \in \mathcal{E} \suchthat j(E_{A,B}) = j\}.
\end{align*}
In other words, $\mathcal{E}_j$ is the subset of all elliptic curves in $\mathcal{E}$ with $j$-invariant equal to $j$. Moreover, let $\mathcal{J}^{\textrm{cm}}$ be the set of all the possible thirteen $j$-invariants of CM elliptic curves over $\Q$, as appearing in Table \ref{tab:fixed_cm_curves_table}. This set is given explicitly by

\begin{align*}
\mathcal{J}^{\textrm{cm}} = \{
& 0, 1728, -3375, 8000, -32768, 54000, 287496, -12288000,\\
& 16581375, -884736, -884736000, -147197952000, -262537412640768000\}.
\end{align*}

For a given $j \in \mathcal{J}^{\textrm{cm}}$, every elliptic curve $E_{A, B} \in \mathcal{E}_j$ has complex multiplication by the order corresponding to the given CM $j$-invariant. Hence, for a CM $j$-invariant $j \in \mathcal{J}^{\textrm{cm}}$ we have
$$
\mathcal{E}_j \subset \mathcal{E}^{\textrm{cm}}.
$$

We will also let $\mathcal{E}_j(X) := \mathcal{E}_j \cap \mathcal{E}(X)$. Then, the classification of CM elliptic curves over $\Q$ implies that we can decompose the families $\mathcal{E}^{\textrm{cm}}$ and $\mathcal{E}^{\textrm{cm}}(X)$ as disjoint unions given by
\begin{align}\label{eq:Ecm-decomposition}
    \mathcal{E}^{\mathrm{cm}} = \coprod_{j \in \mathcal{J}^{\textrm{cm}}} \mathcal{E}_j \quad \text{and} \quad \mathcal{E}^{\mathrm{cm}}(X) = \coprod_{j \in \mathcal{J}^{\textrm{cm}}} \mathcal{E}_j(X)
\end{align}
for every $X > 0$. Thus, we are now ready to prove Theorem \ref{thm:asymptotic-for-curves-with-given-j}, which we restate here for clarity.

\begin{theorem}\label{theo:size_of_cm_elliptic_curves_j_inv}
Let $j \in \Q$. Then
\begin{align}\label{eq:Ej-bound}
    \#\mathcal{E}_j(X) = O(X^{1/m(j)}),
\end{align}
where
\begin{align*}
    m(j) =
        \begin{cases}
            6 & \text{if } j\neq 0, 1728, \\
            3 & \text{if } j = 1728, \\
            2 & \text{if } j = 0.
        \end{cases}
\end{align*}
Moreover, for $j = 0$ and $j = 1728$ we have the asymptotic formulas
\begin{align}\label{eq:E-0-1728-asymptotics}
    \#\mathcal{E}_0(X) = \frac{2}{3^{3/2} \zeta(6)} X^{1/2} + O(X^{1/12})  \quad \text{and} \quad \#\mathcal{E}_{1728}(X) = \frac{2^{1/3}}{ \zeta(4)} X^{1/3} + O(X^{1/12})
\end{align}
as $X \to \infty$. In particular, these asymptotic formulas combined imply that the number of CM elliptic curves in $\mathcal{E}^{\mathrm{cm}}(X)$ satisfies the asymptotic formula
\begin{align}\label{eq:Ecm-asymptotics}
\#\mathcal{E}^{\mathrm{cm}}(X) = \frac{2}{3^{3/2} \zeta(6)} X^{1/2} + \frac{2^{1/3}}{ \zeta(4)} X^{1/3} + O(X^{1/6}),
\end{align}
as $X \to \infty$.
\end{theorem}

\begin{proof}
First, recall that for a curve $E_{A, B} \in \mathcal{E}$, the condition $h^{\mathrm{naive}}(E_{A, B}) \leq X$ is equivalent to the pair of inequalities $|A| \leq \dfrac{X^{1/3}}{2^{2/3}}$ and $|B| \leq \dfrac{X^{1/2}}{3^{3/2}}$. Additionally, the $j$-invariant of $E_{A, B}$ is given by 
$$
j(E_{A, B}) = 1728 \frac{4 A^3}{4A^3 + 27B^2},
$$
and moreover, $j(E_{A, B}) = 0$ if and only if $A = 0$, and $j(E_{A, B}) = 1728$ if and only if $B = 0$. Now, in order to count the number of curves in $\mathcal{E}_j(X)$ for a given $j \in \Q$, we must determine how many integral points $(A, B) \in \Z^2$ satisfy the following conditions:
\begin{enumerate}
    \item The point $(A, B)$ lies within the rectangular box $\mathcal{B}(X)$ determined by the inequalities 
    \begin{align}\label{eq:AB-bounds}
        |A| \leq \dfrac{X^{1/3}}{2^{2/3}} \quad \text{and} \quad |B| \leq \dfrac{X^{1/2}}{3^{3/2}}.
    \end{align}
    \item The discriminant $\Delta_{E_{A, B}} \neq 0$, meaning that the point $(A, B)$ satisfies that $4A^3 + 27B^2 \neq 0$. Equivalently, the point $(A, B)$ must not lie on the cuspidal cubic
    \begin{align}
        y^2 = -\frac{4}{27} x^3.
    \end{align}
    \item The $j$-invariant $j(E_{A, B}) = j$, that is, the point $(A, B)$ must satisfy the equation 
    \begin{align}
    1728 \frac{4 A^3}{4A^3 + 27B^2} = j.
    \end{align}
    \item There is no prime number $p \in \Z$ such that $p^4 \mid A$ and $p^6 \mid B$.
\end{enumerate}

We will now analyze three cases separately, based on whether $j = 0$, $j = 1728$ or $j \neq 0, 1728$.
\vspace{5pt}

$\bullet$ \noindent \textbf{Points $(A, B)$ with $j(E_{A, B}) = 0$.} In this case, we know that $A = 0$. Thus, the set $\mathcal{E}_0(X)$ is in one to one correspondence with the set of points $(0, B) \in \Z^2$ such that $|B| \leq \dfrac{X^{1/2}}{3^{3/2}}$, $B \neq 0$ and there is no prime $p \in \Z$ such that $p^6 \mid B$. In other words, we must count the number of integers $B \in \Z \smallsetminus \{ 0 \}$ in the range $|B| \leq \dfrac{X^{1/2}}{3^{3/2}}$ that are $6$-th power-free. Hence, by Proposition \ref{thm:density-k-free-integers} with $k = 6$, we have
\begin{align}
\mathcal{E}_0(X) = \frac{2}{3^{3/2}\zeta(6)} X^{1/2} + O(X^{1/12}).
\end{align}

$\bullet$ \noindent \textbf{Points $(A, B)$ with $j(E_{A, B}) = 1728$.} In this case, we know that $B = 0$. Thus, the set $\mathcal{E}_{1728}(X)$ is in one to one correspondence with the set of points $(A, 0) \in \Z^2$ such that $|A| \leq \dfrac{X^{1/3}}{2^{2/3}}$, $A \neq 0$ and there is no prime $p \in \Z$ such that $p^4 \mid A$. In other words, we must count the number of integers $A \in \Z \smallsetminus \{ 0 \}$ in the range $|A| \leq \dfrac{X^{1/3}}{2^{2/3}}$ that are $4$-th power-free. Hence, by Proposition \ref{thm:density-k-free-integers} with $k = 4$, we have
\begin{align}
\mathcal{E}_{1728}(X) = \frac{2}{2^{2/3}\zeta(4)} X^{1/3} + O(X^{1/12}) = \frac{2^{1/2}}{\zeta(4)} X^{1/3} + O(X^{1/12}).
\end{align}

$\bullet$ \noindent \textbf{Points $(A, B)$ with $j(E_{A, B}) \neq 0, 1728$.} In this case we have $A \neq 0$ and $B \neq 0$. Moreover, we can rewrite the condition (3) above as follows. Note that
\begin{align*}
j(E_{A, B}) = j \iff 1728 \frac{4A^3}{4A^3 + 27B^2} = j \iff B^2 = \frac{1728 - 4j}{27j} A^3.
\end{align*}
This means that the condition for $E_{A, B}$ having
$j$-invariant equal to $j$, for $j \neq 0,1728$, is equivalent to the point $(A, B) \in(\mathbb{Z} \smallsetminus \{0\}) \times(\mathbb{Z} \smallsetminus \{0\})$ lying on the cuspidal cubic
\begin{align}\label{eq:cuspidal-cubic}
y^2 = \frac{1728 - 4j}{27j} x^3.
\end{align}
and having $|A| \leq  \dfrac{X^{1/3}}{2^{2/3}}$, because of the bounds \eqref{eq:AB-bounds}. Using Lemma \ref{lemma:technical_lemma} with $a=a_j= \frac{1728 - 4j}{27j}$ and $T=\dfrac{X^{1/3}}{2^{2/3}}$, we conclude that the number of these points is $O(X^{1/6})$. In particular, since we didn't take into account the condition (4) about the prime divisors of $A$ and $B$, the number of points $(A, B)$ corresponding to curves $E_{A, B} \in \mathcal{E}_j(X)$ is also $O(X^{1/6})$.

This completes the proof of the asymptotic formulas \eqref{eq:Ej-bound} and \eqref{eq:E-0-1728-asymptotics}. Finally, using this in combination with the decomposition
\begin{align*}
\quad \mathcal{E}^{\mathrm{cm}}(X) = \coprod_{j \in \mathcal{J}^{\textrm{cm}}} \mathcal{E}_j(X),
\end{align*}
we see that 
\begin{align*}
\# \mathcal{E}^{\textrm{cm}}(X) &= \sum_{j \in \mathcal{J}^{\textrm{cm}}} \# \mathcal{E}_j(X)\\
&= \# \mathcal{E}_0(X) + \# \mathcal{E}_{1728}(X) + \sum_{\substack{j \in \mathcal{J}^{\textrm{cm}}\\j \neq 0, 1728}} \#\mathcal{E}_j(X)\\
&= \frac{2}{3^{3/2} \zeta(6)} X^{1/2} + \frac{2^{1/3}}{ \zeta(4)} X^{1/3} + O(X^{1/6}),
\end{align*}
and this finishes the proof of the theorem.
\end{proof}

As consequences of this theorem we prove Theorem \ref{thm:density-of-cm-elliptic-curves} and Theorem \ref{thm:density-j-0} next.

\begin{theorem}\label{theo:density_of_cm_elliptic_curves}
The natural density of the set of CM elliptic curves $\mathcal{E}^{\mathrm{cm}}$ in the family $\mathcal{E}$ is
\begin{align}
d(\mathcal{E}^{\mathrm{cm}}) = \lim_{X \to \infty} \frac{\#\mathcal{E}^{\textrm{cm}}(X)}{\# \mathcal{E}(X)} = 0.
\end{align}
\end{theorem}

\begin{proof}
Using the asymptotic formulas \eqref{eqn:elliptic-count} and \eqref{eq:Ecm-asymptotics} we have
\begin{align*}
d(\mathcal{E}^{\mathrm{cm}}) &= \lim_{X\to \infty} \frac{\#\mathcal{E}^{\mathrm{cm}}(X)}{\#\mathcal{E}(X)} = \lim_{X \to \infty} \frac{ \frac{2}{3^{3/2} \zeta(6)} X^{1/2} + \frac{2^{1/3}}{ \zeta(4)} X^{1/3} + O(X^{1/6}) }{ \frac{2^{4/3}}{3^{3/2} \zeta(10)} X^{5/6} + O(X^{7/12}) }\\
&= \lim_{X\to \infty} \frac{1}{X^{1/3}} \frac{ \frac{2}{3^{3/2}\zeta(6)} + \frac{2^{1/3}}{\zeta(4)} X^{-1/6} + O(X^{-1/3}) }{ \frac{2^{4/3}}{3^{3/2} \zeta(10)} + O(X^{-1/4}) } = 0.
\end{align*}
\end{proof}

\begin{theorem}\label{theo:distribution_of_cm_elliptic_curves}
The elliptic curves with $j$-invariant 0 comprise $100\%$ of all the CM elliptic curves in $\mathcal{E}^{\textrm{cm}}$. More precisely, the natural density of the family $\mathcal{E}_0$ in $\mathcal{E}^{\mathrm{cm}}$ equals $1$, that is
\begin{align*}
d(\mathcal{E}_0) = \lim_{X\to \infty} \frac{\#\mathcal{E}_0(X)}{\#\mathcal{E}^{\mathrm{cm}}(X)} = 1.
\end{align*}
\end{theorem}

\begin{proof}
Using the asymptotic formulas from Theorem \ref{theo:size_of_cm_elliptic_curves_j_inv} we have
\begin{align*}
    d(\mathcal{E}_0) = \lim_{X\to \infty} \frac{\#\mathcal{E}_0(X)}{\#\mathcal{E}^{\mathrm{cm}}(X)} &=\lim_{X\to \infty} \frac{\frac{2}{3^{3/2} \zeta(6)} X^{1/2} + O(X^{1/12})}{\frac{2}{3^{3/2} \zeta(6)} X^{1/2} + \frac{2^{1/3}}{ \zeta(4)} X^{1/3} + O(X^{1/6})} \\
        &= \lim_{X\to \infty} \frac{\frac{2}{3^{3/2} \zeta(6)} + O(X^{-5/12})}{\frac{2}{3^{3/2} \zeta(6)} + \frac{2^{1/3}}{ \zeta(4)} X^{-1/6} + O(X^{-1/3})}  =  1.
\end{align*}
\end{proof}

\section{Isomorphism classes of elliptic curves and twisting}\label{section:twists}

In this section, we revisit fundamental definitions and results regarding twists of elliptic curves, as outlined in \cite[Chapter X]{Sil09}. This provides an alternative method of indexing the $K$-isomorphism classes of elliptic curves over a perfect field $K$. Subsequently, we employ this approach to enumerate CM elliptic curves over $\Q$ in a different way than the one used in Section \ref{section:distribution-of-cm-elliptic-curves-over-E} and we study their distribution in Section \ref{section:distribution-of-cm-elliptic-curves-using-twists}.

First, if $C/K$ is a smooth projective curve over a perfect field $K$, then the group of all $\widebar{K}$-isomorphisms $C \longrightarrow C$, called the \textit{isomorphism group} of $C$, is denoted by $\operatorname{Isom}(C)$. Then, a \textit{twist} of $C/K$ is a smooth curve $C^\prime/K$ that is isomorphic to $C$ over $\widebar{K}$. Moreover, two twists are treated as \textit{equivalent} if they are isomorphic over $K$, and the set of all twists of $C/K$, up to $K$-isomorphism, is denoted by $\operatorname{Twist}(C/K)$.

Now we briefly explain how $\operatorname{Twist}(C/K)$ can be identified with a particular pointed cohomology set. Given a twist $C^{\prime}/K$ and an isomorphism $\phi \colon C^{\prime} \longrightarrow C$ defined over $\widebar{K}$, let $\xi \colon G_{\widebar{K}/K} \longrightarrow \operatorname{Isom}(C)$ be given by $\xi(\sigma) = \xi_{\sigma} := \phi^{\sigma} \phi^{-1}$, for every $\sigma \in G_{\widebar{K}/K}$. It turns out that the map $\xi$ satisfies the relations defining a 1-cocycle, that is, we have
$$
\xi_{\sigma \tau}=\left(\xi_\sigma\right)^\tau \xi_\tau \quad \text { for all } \sigma, \tau \in G_{\widebar{K} / K} .
$$
In particular, we denote by $\{ \xi \}$ the associated cohomology class in $H^1(G_{\widebar{K} / K}, \operatorname{Isom}(C) )$. Moreover, it is known that the cohomology class $\{ \xi \}$ is completely determined by the $K$-isomorphism class of $C^{\prime}/K$ and that it is independent of the choice of isomorphism $\phi$. Therefore, this allows one to define a map
\begin{align}\label{eqn:twist-cohomology-bijection}
\operatorname{Twist}(C/K) &\longrightarrow H^1(G_{\widebar{K} / K}, \operatorname{Isom}(C) )\\
C^{\prime}/K &\longmapsto \{ \xi \}.\notag
\end{align}
Furthermore, this map is a bijection. These facts can be found in \cite[Theorem X.2.2]{Sil09}.

Now, if $E/K$ is an elliptic curve then the automorphism group of $E$, denoted by $\operatorname{Aut}(E)$, is the subgroup of $\operatorname{Isom}(E)$ consisting of all isomorphisms $E \longrightarrow E$ that also map the origin $O \in E$ to itself. Then, the inclusion $\operatorname{Aut}(E) \subset \operatorname{Isom}(E)$ induces an inclusion
$$
H^1 (G_{\widebar{K} / K}, \operatorname{Aut}(E) ) \subset H^1 (G_{\widebar{K} / K}, \operatorname{Isom}(E) ) .
$$
By the bijection \eqref{eqn:twist-cohomology-bijection}, the pointed cohomology set $H^1 (G_{\widebar{K} / K}, \operatorname{Isom}(E) )$ is identified with $\operatorname{Twist}(E/K)$ and hence we denote the cohomology group $H^1 (G_{\widebar{K} / K}, \operatorname{Aut}(E) )$ by $\operatorname{Twist}((E, O)/K)$. Also, it is known that if $C/K \in \operatorname{Twist}((E, O)/K)$, then $C/K$ can be given the structure of an elliptic curve over $K$ and conversely, if $E^{\prime} / K$ is an elliptic curve that is isomorphic to $E$ over $K$, then $E^{\prime} / K$ represents an element of $\operatorname{Twist}((E, O) / K)$. These facts can be found in \cite[Proposition X.5.3]{Sil09}.

Furthermore, if $\operatorname{char}(K) \neq 2,3$, then the elements of the group $\operatorname{Twist}((E, O) / K)$ can be explicitly described, as the two following results show. These are stated verbatim from \cite{Sil09}.

\begin{proposition}[{\cite[Proposition X.5.4]{Sil09}}]\label{prop:twist-equations}
Assume that $\operatorname{char}(K) \neq 2,3$, and let
$$
n = 
\begin{cases}
2 & \text { if } j(E) \neq 0,1728, \\ 
4 & \text { if } j(E) = 1728, \\ 
6 & \text { if } j(E) = 0.
\end{cases}
$$

Then $\operatorname{Twist}((E, O) / K)$ is canonically isomorphic to $K^\times /\left(K^\times\right)^n$.
More precisely, choose a Weierstrass equation
$$
E: y^2=x^3+A x+B
$$
for $E / K$, and let $D \in K^\times$. Then the elliptic curve $E_D \in \operatorname{Twist}((E, O) / K)$ corresponding to $D\left(\bmod \left(K^\times\right)^n\right)$ has Weierstrass equation
\begin{itemize}
\item[(i)] $E_D: y^2=x^3+D^2 A x+D^3 B \quad$ if \quad $j(E) \neq 0,1728$,
\item[(ii)] $E_D: y^2=x^3+D A x$ \quad if \quad $j(E)=1728$ (so $B = 0$),
\item[(iii)] $E_D: y^2=x^3+D B$ \quad if \quad $j(E) = 0$ (so $A = 0$).
\end{itemize}
\end{proposition}

\begin{corollary}[{\cite[Corollary X.5.4.1]{Sil09}}]\label{cor:corollary_silverman}
Define an equivalence relation on the set $K \times K^\times$ by
$$
(j, D) \sim\left(j^{\prime}, D^{\prime}\right) \quad \text { if } \quad j=j^{\prime} \quad \text { and } \quad D / D^{\prime} \in\left(K^\times\right)^{n(j)},
$$
where $n(j)=2$ (respectively $4$, respectively $6$) if $j \neq 0,1728$ (respectively $j = 1728$, respectively $j = 0$). Then the $K$-isomorphism classes of elliptic curves $E / K$ are in one-to-one correspondence with the elements of the quotient
$$
\frac{K \times K^\times}{\sim} .
$$
\end{corollary}

We will apply these two results when $K = \Q$. In this case we note the following.

\begin{lemma}\label{lem:quotient_to_integers_lemma}
Let $n \in \Z_{\geq 1}$ be an even integer and let $r \in \Q^{\times}$. Then there exists a unique $n$-th power-free integer $z_r \in \Z^{\times}$ such that $r \cdot (\Q^{\times})^n = z_r \cdot (\Q^{\times})^n$. In other words, every coset in the quotient group $\Q^{\times}/(\Q^{\times})^n$ contains a unique representative that is an $n$-th power-free integer. Thus, there is a bijection
\begin{align}\label{eqn:n-th-power-quotient-bijection}
\Q^{\times}/(\Q^{\times})^n &\longrightarrow \Z^{\times}/(\Q^{\times})^n\\
r \cdot (\Q^{\times})^n &\longmapsto z_r \cdot (\Q^{\times})^n .\notag
\end{align}
Here we denote $\Z^{\times} := \Z \setminus \{0\}$.
\end{lemma}

\begin{proof}
Let $r = a/b \in \Q^{\times}$ for some $a, b \in \Z^{\times}$ with $a$ and $b$ both relatively prime and $n$-th power-free, that is, both are not divisible by the $n$-th power of some integer greater than $1$. We can safely assume this, as in the case where either $a$ or $b$ is not $n$-th power free, the largest perfect $n$-th powers dividing them would be absorbed within the quotient $\Q^{\times}/(\Q^{\times})^n$.

Then we have
$$
r \cdot (\Q^{\times})^n = \frac{a}{b} \cdot (\Q^{\times})^n = \frac{a b^{n-1}}{b^n} \cdot (\Q^{\times})^n = ab^{n - 1} \cdot (\Q^{\times})^n.
$$
Now, we have $ab^{n-1} \in \Z^{\times}$ but it may not be $n$-th power-free. Thus we remove the largest $n$-th power $c^n$ dividing it and write $ab^{n-1} = z_r \cdot c^n$ with $z_r \in \Z^{\times}$ and $n$-th power-free. Then clearly 
$$
r \cdot (\Q^{\times})^n = z_r \cdot (\Q^{\times})^n.
$$
To prove uniqueness, suppose that $w_r \in \Z^{\times}$ is $n$-th power-free and that $z_r \cdot (\Q^{\times})^n = w_r \cdot (\Q^{\times})^n.$ Thus there are relatively prime integers $d, e \in \Z^{\times}$ such that
$$
z_r = w_r \cdot \left(\frac{d}{e}\right)^n = w_r \cdot \frac{d^n}{e^n}.
$$
Hence $z_r e^n = w_r d^n$. This implies that $e^n$ divides $w_r$ and similarly that $d^n$ divides $z_r$ because $\gcd(d^n, e^n) = 1$ since $\gcd(d, e) = 1$. Given that both $z_r$ and $w_r$ are $n$-th power-free and that $n$ is even, this implies that $d^n = e^n = 1$ and therefore $z_r = w_r$.
\end{proof}

The following result shows how the naive height behaves after twisting.

\begin{lemma}\label{lem:naive_height_twist}
Let $E \colon y^2 = x^3 + Ax + B$ be an elliptic curve over $\Q$ with $A, B \in \Z$ and suppose that $D \in \Z^{\times} = \Z \smallsetminus \{ 0 \}$ is a nonzero integer. Also, let $n = n(E) \in \{ 2, 4, 6\}$, according to the value of $j(E)$, as defined in Proposition \ref{prop:twist-equations}. Then, if $E_D$ denotes the twist of $E/\Q$ corresponding to the coset $D \pmod{(\Q^{\times})^n}$ as in Proposition \ref{prop:twist-equations}, the naive height of $E_D$ satisfies
$$
h^{\textrm{naive}}(E_D) =|D|^{m} h^{\textrm{naive}}(E),
$$ 
where $m$ is given by
\begin{align}
m = 
\begin{cases}
6 &\text{if $j(E) \neq 0, 1728$},\\
3 &\text{if $j(E) = 1728$},\\
2 &\text{if $j(E) = 0$}.
\end{cases}
\end{align}
\end{lemma}

\begin{proof}
The proof is just a short calculation using the definition 
$$
h^{\textrm{naive}}(E) := \max{\{ 4|A|^3, 27 |B|^2 \}}
$$ 
of the naive height of an elliptic curve $E \colon y^2 = x^3 + Ax + B$ and the explicit equations for the twists $E_D$, as given in Proposition \ref{prop:twist-equations}.
\end{proof}

\begin{remark}
Observe that the values of $m$ appearing in the previous lemma coincide with the values of $m(j)$ that appear in Theorem \ref{theo:size_of_cm_elliptic_curves_j_inv}.    
\end{remark}

\section{The distribution of CM elliptic curves over \texorpdfstring{$\Q$}{Q} in the family \texorpdfstring{$\mathcal{ET}$}{ET}}\label{section:distribution-of-cm-elliptic-curves-using-twists}

In this section we study again the distribution of CM elliptic curves over $\Q$, but this time we use the theory of twists that was summarized in Section \ref{section:twists} in order to produce a different family of representatives for the $\Q$-isomorphism classes of CM elliptic curves over $\Q$ than the family $\mathcal{E}^{\textrm{cm}}$ that was used in the introduction and in Section \ref{section:distribution-of-cm-elliptic-curves-over-E}.

For each $j \in \Q$, let $E^{j}$ be any \textit{fixed} elliptic curve in short Weierstrass form $E^j \colon y^2 = x^3 + A_j x + B_j$ with $A_j, B_j \in \Z$ and having $j$-invariant $j(E^j) = j$. Then, for each of the curves $E^j$ we define the corresponding family of twists
\begin{equation}\label{eq:e_dk_f_family_definition}
    \mathcal{ET}_j := \{ E^j_D \suchthat D \in \Z^{\times}/(\Q^{\times})^{n(j)}\},
\end{equation}
where $n(j) = 2, 4$ or $6$, according to whether $j \neq 0, 1728$, $j = 1728$, or $j = 0$, respectively, as was defined in Corollary \ref{cor:corollary_silverman}. Moreover, we only consider representatives $D \in \Z^{\times}$ for the cosets as a consequence of Lemma \ref{lem:quotient_to_integers_lemma}. In particular, when taken in combination with Corollary \ref{cor:corollary_silverman} we have that every elliptic curve defined over $\Q$ and with $j$-invariant equal to $j$ is isomorphic over $\Q$ to a unique elliptic curve in the family $\mathcal{ET}_j$. 

Next, recall that $\mathcal{J}^{\textrm{cm}}$ was defined in \eqref{eqn:set-of-cm-j-invariants} as the set of thirteen CM $j$-invariants of CM elliptic curves over $\Q$. Then, we define the family 
\begin{equation}
    \mathcal{ET}^{\mathrm{cm}} := \coprod_{j \in \mathcal{J}^{\textrm{cm}}} \mathcal{ET}_j.
\end{equation}
In particular, if $E/\Q$ is any CM elliptic curve, then $E$ is $\Q$-isomorphic to a unique elliptic curve in the family $\mathcal{ET}$.

\begin{remark}
We have chosen the notations $\mathcal{ET}$, $\mathcal{ET}_j$ and $\mathcal{ET}^{\mathrm{cm}}$ to reflect the fact that the families consist of the twists of certain fixed elliptic curves.
\end{remark}

\begin{remark}
By Proposition \ref{prop:twist-equations} every elliptic curve in $\mathcal{ET}_j$ is in short Weierstrass form so its naive height is well defined.
\end{remark}

\begin{remark}
The definitions of the families $\mathcal{ET}_j$ and $\mathcal{ET}^{\mathrm{cm}}$ depend on the choice of elliptic curves $E^j$ for different values of $j$. Therefore, a notation like $\mathcal{ET}(j, E^j)$ might be more appropriate. However, to avoid overcomplicating the notation, we decided to use the former.
\end{remark}

Then, as we did with the families studied in the introduction and in Section \ref{section:distribution-of-cm-elliptic-curves-over-E}, for each $X > 0$ we define the \textit{finite} subsets of $\mathcal{ET}$ given by
\begin{align*}
\mathcal{ET}^{\mathrm{cm}}(X) := \{ E \in \mathcal{ET}^{\mathrm{cm}} \suchthat h^{\textrm{naive}}(E) \leq X\},  
\end{align*}
and for each CM $j$-invariant $j \in \mathcal{J}^{\textrm{cm}}$ the subsets
\begin{align*}
\mathcal{ET}_j(X) := \{ E \in \mathcal{ET}_j \suchthat h^{\textrm{naive}}(E) \leq X\}.
\end{align*}
It then follows that we have the disjoint union decomposition
\begin{equation}
    \mathcal{ET}^{\mathrm{cm}}(X) = \coprod_{j \in \mathcal{J}^{\textrm{cm}}} \mathcal{ET}_j(X).
\end{equation}

Then, using these definitions, we want to study how the CM elliptic curves are distributed inside the family $\mathcal{ET}^{\textrm{cm}}$ among the thirteen different CM $j$-invariants. In particular, we have the following theorem.

\begin{theorem}\label{thm:cm-elliptic-curves-asymptotic-twist-count}
Let $j \in \Q$. Let $n(j) = 2, 4$ or $6$ and $m(j) = 6, 3$ or $2$ according to whether $j \neq 0, 1728$, $j = 1728$ or $j = 0$, respectively. Then for $X > 0$ we have
\begin{align}\label{eqn:ETj-asymptotic}
\#\mathcal{ET}_j(X) = C(j) X^{1/m(j)} + O(X^{1/12})
\end{align}
as $X \to \infty$, where 
\begin{align}
C(j) := \frac{2}{\zeta(n(j))} \left(\frac{1}{h^{\textrm{naive}}(E^j)}\right)^{1/m(j)}.
\end{align}
Here $\zeta(s)$ denotes the Riemann zeta function and $E^j$ is an elliptic curve in short Weierstrass form $E^j \colon y^2 = x^3 + A_jX + B_j$ with $A_j, B_j \in \Z$ having $j$-invariant equal to $j$. Moreover, we have
\begin{align}\label{eqn:ET-asymptotic}
\#\mathcal{ET}^{\mathrm{cm}}(X) =  C(0)X^{1/2} + C(1728)X^{1/3} + \left( \sum_{\substack{j \in \mathcal{J}^{\mathrm{cm}}\\j \neq 0, 1728}} C(j) \right) X^{1/6} + O(X^{1/12}).
\end{align}
\end{theorem}

\begin{proof}
Let $j \in \Q$ be fixed and let $X>0$. Then, using the formula from Lemma \ref{lem:naive_height_twist}, for a twist $E^j_D \in \mathcal{ET}^{\mathrm{cm}}_j$ we have

\begin{align*}
    h^{\textrm{naive}}(E^j_D) \leq X &\iff |D|^{m(j)} h^{\textrm{naive}}(E^j) \leq X \\
    &\iff |D|^{m(j)} \leq \frac{1}{h^{\textrm{naive}}(E^j)} X \\
    &\iff |D| \leq \left(\frac{1}{h^{\textrm{naive}}(E^j)}\right)^{1/m(j)} X^{1/m(j)}.
\end{align*}
Then, since 
\begin{align*}
    \mathcal{ET}_j(X) = \{ E^j_D \suchthat D \in \Z^{\times}/(\Q^{\times})^{n(j)} \, \text{and $h^{\textrm{naive}}(E_D^j) \leq X$}\},
\end{align*}
we conclude that
\begin{align}\label{eq:number_elements_e_dk_f}
    \# \mathcal{ET}_j(X)= \left\{ D \in \Z^{\times} \suchthat \text{$D$ is $n(j)$-free and } |D| \leq \left(\frac{1}{h^{\textrm{naive}}(E^j)}\right)^{1/m(j)} X^{1/m(j)}  \right\}.
\end{align}
Now, recall from Proposition \ref{thm:density-k-free-integers} that for $X > 0$ and $k \geq 2$ an integer, $Q_k(X)$ denotes the number of positive $k$-free integers $n$ such that $n \leq X$. Then, it follows from \eqref{eq:number_elements_e_dk_f} that
\begin{align*}
\# \mathcal{ET}_j(X) = 2 Q_{n(j)} \left( \left(\frac{1}{h^{\textrm{naive}}(E^j)}\right)^{1/m(j)} X^{1/m(j)}  \right).
\end{align*}

Hence, Proposition \ref{thm:density-k-free-integers} implies that

\begin{align}\label{eq:asymptotic_bound_E}
    \# \mathcal{ET}_j(X) = C(j) X^{1/m(j)} + O(X^{1/m(j)n(j)}) = C(j) X^{1/m(j)} + O(X^{1/12}),
\end{align} 
where the constant $C(j)$ is given by
\begin{align}\label{eq:constant_c_definition}
    C(j) = \frac{2}{\zeta(n(j))} \left(\frac{1}{h^{\textrm{naive}}(E^j)}\right)^{1/m(j)}
\end{align}
and we used the fact that $m(j)n(j)=12$, which follows from the definitions of $m(j)$ and $n(j)$ given in the statement of the theorem. This proves \eqref{eqn:ETj-asymptotic}. Finally, to prove \eqref{eqn:ET-asymptotic} recall that we have
\begin{align*}
    \mathcal{ET}^{\mathrm{cm}}(X) = \coprod_{j \in \mathcal{J}^{\textrm{cm}}} \mathcal{ET}_j(X).
\end{align*}
Hence, combining this decomposition with the asymptotic formulas \eqref{eqn:ETj-asymptotic} for all $j \in \mathcal{J}^{\textrm{cm}}$ gives the desired asymptotic. This completes the proof of the theorem.
\end{proof}

\begin{remark}
For $j \in \mathcal{J}^{\mathrm{cm}}$, Table \ref{table:densities-cm-j-invariants-ET} shows the proportions of CM elliptic curves with naive height $h^{\textrm{naive}}(E) \leq 10^{10}$ among the thirteen families $\mathcal{ET}_j$ within $\mathcal{ET}^{\mathrm{cm}}$. The reader can compare the similarity of these values with the ones presented in Table \ref{table:densities-cm-j-invariants} for the families $\mathcal{E}_j$ within $\mathcal{E}^{\mathrm{cm}}$.
\end{remark}

\begin{table}[H]
{\tabulinesep=1.2mm
\begin{tabu}{c c c c c} \hline
$d_K$ & $f$ & $j$-invariant & $\mathcal{ET}_{j}(10^{10})$ & $\dfrac{\mathcal{ET}_{j}(10^{10})}{\mathcal{ET}^{\text{cm}}(10^{10})}$ \\ \hline
\multirow{3}{*}{$-3$} & $1$ & $0$ & $38152$ & $0.934777282$ \\
& $2$ & $2^4 \cdot 3^3 \cdot 5^3$ & $14$ & $0.00034302$ \\
& $3$ & $-2^{15} \cdot 3 \cdot 5^3$ & $6$ & $0.000147008$ \\ \hline
\multirow{2}{*}{$-4$} & $1$ & $2^6 \cdot 3^3$ & $2606$ & $0.063850639$ \\
& $2$ & $2^3 \cdot 3^3 \cdot 11^3$ & $18$ & $0.000441025$ \\\hline
\multirow{2}{*}{$-7$} & $1$ & $-3^3 \cdot 5^3$ & $0$ & $0$ \\
& $2$ & $3^3 \cdot 5^3 \cdot 17^3$ & $2$ & $0.000049003$ \\ \hline
$-8$ & $1$ & $2^6 \cdot 5^3$ & $10$ & $0.000245014$ \\
$-11$ & $1$ & $-2^{15}$ & $2$ & $0.000049003$ \\
$-19$ & $1$ & $-2^{15} \cdot 3^3$ & $4$ & $0.000098006$ \\
$-43$ & $1$ & $-2^{18} \cdot 3^3 \cdot 5^3$ & $0$ & $0$ \\
$-67$ & $1$ & $-2^{15} \cdot 3^3 \cdot 5^3 \cdot 11^3$ & $0$ & $0$ \\
$-163$ & $1$ & $-2^{18} \cdot 3^3 \cdot 5^3 \cdot 23^3 \cdot 29^3$ & $0$ & $0$ \\ \hline \\
\end{tabu}}
\caption{The proportions of elliptic curves with CM by the different orders of class number 1 in the family $\mathcal{ET}^{\mathrm{cm}}$ and with naive height $h^{\textrm{naive}}(E) \leq 10^{10}$.}
\label{table:densities-cm-j-invariants-ET}
\end{table}

\begin{remark}
We note that in equation \eqref{eqn:ETj-asymptotic}, we obtained a main term of order $X^{1/6}$ for $j \neq 0, 1728$ and an error term of $O(X^{1/12})$ for the counting function $\#\mathcal{ET}_j(X)$. In contrast, in equation \eqref{eqn:Ej-asymptotic}, we only proved that the counting function satisfies $\#\mathcal{E}_j(X) = O(X^{1/6})$. This difference is reflected in the corresponding asymptotic formulas \eqref{eqn:ECM-asymptotic} for $\#\mathcal{E}^{\mathrm{cm}}(X)$ and \eqref{eqn:ET-asymptotic} for $\#\mathcal{ET}^{\mathrm{cm}}(X)$, where for the latter, we obtain a tertiary term of order $X^{1/6}$.
\end{remark}

Finally, as a direct consequence of Theorem \ref{thm:cm-elliptic-curves-asymptotic-twist-count} we have the following.

\begin{theorem}\label{theo:distribution_of_cm_elliptic_curves_with_twists}
The elliptic curves with $j$-invariant 0 comprise $100\%$ of all the elliptic curves in the family $\mathcal{ET}^{\mathrm{cm}}$. In other words, the natural density of the subfamily $\mathcal{ET}_0$ of elliptic curves with $j$-invariant $0$ inside the family $\mathcal{ET}^{\mathrm{cm}}$ equals $1$. More explicitly,
\begin{align*}
    d(\mathcal{ET}_0) = \lim_{X\to \infty} \frac{\#\mathcal{ET}_0(X)}{\#\mathcal{ET}^{\mathrm{cm}}(X)} = 1.
\end{align*}
\end{theorem}

\section{Computational methods and code description}

The \texttt{Python} code used for the computations made in this article can be found in the \texttt{GitHub} repository \cite{BSCM24}. This code computes the distribution of $j$-invariants of CM elliptic curves with naive height less than $X$ by counting the appearances of $j$-invariants in each family $\mathcal{E}$ and $\mathcal{ET}$.

Here is a detailed explanation of what each function in the code does, along with its role in the overall analysis.

\subsection*{1. \texttt{disc(a, b)}} Calculates the discriminant of the elliptic curve \( E_{a,b}: y^2 = x^3 + ax + b \).

\subsection*{2. \texttt{naive\_height(a, b)}} Computes the \textit{naive height} of the elliptic curve \( E_{a,b} \).

\subsection*{3. \texttt{prime\_condition(a, b)}} Checks the primality condition on the integers \( a \) and \( b \), ensuring there is no prime \( p \) such that \( p^4 \) divides \( a \) and \( p^6 \) divides \( b \). The function analyzes the prime factorizations of \( a \) and \( b \) and returns \texttt{True} if the condition is satisfied, otherwise \texttt{False}.

\subsection*{4. \texttt{Get\_Distribution\_J\_invariants\_in\_family\_E(Xmax)}}
Computes the distribution of \( j \)-invariants in the family $\mathcal{E}$ for elliptic curves with naive height \( \leq X_{\text{max}} \). \\
\textbf{Steps}:
\begin{enumerate}
    \item Iterates over integers \( a \) and \( b \) within the range dictated by \( X_{\text{max}} \).
    \item For each pair \( (a, b) \), checks that the discriminant is non-zero and that the curve satisfies the prime condition.
    \item Computes the \( j \)-invariant for valid curves and counts how often specific \( j \)-invariants appear.
\end{enumerate}
This function finds all elliptic curves in the family $\mathcal{E}$ up to the given height bound and generates a dictionary of \( j \)-invariants and their frequency. The pseudocode for this function can be found in Algorithm \ref{algo:get_dist_E}.

\subsection*{5. \texttt{is\_k\_power\_free(n, k)}} Checks if the integer \( n \) is \textit{k-power-free}, meaning no prime factor of \( n \) is raised to an exponent of \( k \) or more in its factorization.

\subsection*{6. \texttt{get\_power\_free\_integers(X, k)}} Generates a list of integers \( d \) such that \( d \) is \( k \)-power-free and \( |d| \leq X \).

\subsection*{7. \texttt{Get\_Distribution\_J\_invariants\_in\_family\_ET(Xmax)}}
Computes the distribution of \( j \)-invariants the family $\mathcal{ET}$, focusing on curves that have \textit{complex multiplication} (CM). \\
\textbf{Steps}:
\begin{enumerate}
    \item Starts with a list of fixed elliptic curves corresponding to known \( j \)-invariants.
    \item For each fixed curve, scales the coefficients by integers \( d \), which are filtered to be \( k \)-power-free, where \( k \) depends on the \( j \)-invariant.
    \item For each scaled elliptic curve, checks whether the curve has CM and counts the number of CM curves for each \( j \)-invariant.
\end{enumerate}
The pseudocode for this function can be found in Algorithm \ref{algo:get_dist_ET}.

\begin{remark}
This code can be run sequentially or in parallel. We recommend to use the parallelized version for larger values of $X$. The code imports the SageMath library \texttt{sage} to compute factorizations and $j$-invariants. 
\end{remark}

\begin{algorithm}
\caption{Computes a dictionary containing a list of key-value pairs of the form $(j, \mathcal{E}_j(X_{\text{max}}))$ for every $j \in \mathcal{J}^{\mathrm{cm}}$.}\label{algo:get_dist_E}
\begin{algorithmic}[1]
\Procedure{Get\_Distribution\_J\_invariants\_in\_family\_E}{$X_{\text{max}}$}
    \State Initialize $a_{\text{bound}}, b_{\text{bound}}$ based on $X_{\text{max}}$
    \ForAll{$a \in [-a_{\text{bound}}, a_{\text{bound}}]$}
        \ForAll{$b \in [-b_{\text{bound}}, b_{\text{bound}}]$}
            \If{\Call{disc}{$a,b$} $\neq 0$ and \Call{prime\_condition}{$a,b$} is \texttt{True}}
                \State Add $(a,b)$ to list of valid elliptic curves
            \EndIf
        \EndFor
    \EndFor
    \ForAll{valid elliptic curves $(a,b)$}
        \State Compute $j$-invariant
        \If{$j$ is in the predefined list of possible $j$-invariants}
            \State Add to occurrences of $j$
        \EndIf
    \EndFor
    \State Return frequency distribution of $j$-invariants
\EndProcedure
\end{algorithmic}
\end{algorithm}

\begin{algorithm}
\caption{Computes a dictionary containing a list of key-value pairs of the form $(j, \mathcal{ET}_j(X_{\text{max}}))$ for every $j \in \mathcal{J}^{\mathrm{cm}}$.}\label{algo:get_dist_ET}
\begin{algorithmic}[1]
\Procedure{Get\_Distribution\_J\_invariants\_in\_family\_ET}{$X_{\text{max}}$}
    \ForAll{fixed $j$-invariants in Family ET}
        \State Determine $(a_j, b_j)$ and height for the fixed curve
        \State Compute scaling bounds for integers $d$
        \State Filter power-free integers for scaling
        \ForAll{valid scalings $d$}
            \State Scale $(a_j, b_j)$ and compute new curve coefficients
            \If{Elliptic curve has CM}
                \State Add to occurrences of $j$
            \EndIf
        \EndFor
    \EndFor
    \State Return frequency distribution of CM curves
\EndProcedure
\end{algorithmic}
\end{algorithm}

\section{Acknowledgments}
The authors thank the Centro de Investigación en Matemática Pura y Aplicada (CIMPA) and the School of Mathematics of the University of Costa Rica for their administrative help and support during this project. This paper was written as part of research project 821-C1-218, led by the first author and registered with the Vicerrectoría de Investigación of the University of Costa Rica. We also thank Álvaro Lozano-Robledo for his helpful comments.

% \newpage
\bibliography{densities-biblio.bib}
\bibliographystyle{alphaurl}

\end{document}